\documentclass[12pt,oneside,english,reqno]{amsart}
\usepackage[T1]{fontenc}
\usepackage[utf8]{inputenc}
\usepackage[a4paper]{geometry}
\usepackage{mathrsfs}
\usepackage{amstext}
\usepackage{amsthm}
\usepackage{amssymb}
\usepackage{color}
\usepackage[dvipsnames]{xcolor}
\usepackage[normalem]{ulem}

\usepackage[colorlinks,linkcolor=blue, urlcolor  = blue,citecolor=blue,pagebackref,hypertexnames=false, breaklinks]{hyperref}
\usepackage{dsfont}
\usepackage{enumerate}
\usepackage{verbatim} 
\usepackage{stmaryrd}
\usepackage{enumitem}
\usepackage{graphicx}
\usepackage{todonotes}
\usepackage{xcolor}
\usepackage{cleveref}

\usepackage{shuffle}

\usepackage{caption}
\usepackage{subcaption}

\usepackage{xstring}
\newcommand{\perm}[1]{%
  \textcolor{black}{%
  \mathtt{%
  [%
    \StrLen{#1}[\stringLength]
    \foreach \n in {1,...,\stringLength}{
      \StrChar{#1}{\n}
      \ifnum\n<\stringLength
      \,
      \fi
    }%
  ]%
  }%
}%
}

\makeatletter
\numberwithin{equation}{section}
\numberwithin{figure}{section}
\theoremstyle{plain}
\newtheorem{thm}{\protect\theoremname}[section]
\theoremstyle{definition}
\newtheorem{defn}[thm]{\protect\definitionname}
\newtheorem{nota}[thm]{Notation}
\theoremstyle{remark}
\newtheorem{rem}[thm]{\protect\remarkname}
\theoremstyle{plain}
\newtheorem{lem}[thm]{\protect\lemmaname}
\theoremstyle{plain}
\newtheorem{prop}[thm]{\protect\propositionname}
\theoremstyle{plain}

\theoremstyle{plain}
\newtheorem{example}[thm]{\protect\examplename}

\makeatother

\usepackage{babel}
\providecommand{\corollaryname}{Corollary}
\providecommand{\definitionname}{Definition}
\providecommand{\lemmaname}{Lemma}
\providecommand{\propositionname}{Proposition}
\providecommand{\remarkname}{Remark}
\providecommand{\theoremname}{Theorem}
\providecommand{\examplename}{Example}

\newcommand{\hdd}{\hat{\mathrm{d}}}
\newcommand{\hw}{\hat{w}}
\newcommand{\id}{\mathsf{id}} 

\newcommand{\ott}{[0,T]}

\newcommand{\sh}{\textnormal{Sh}}
\newcommand{\sq}{\square}

\newcommand{\cA}{\mathcal{A}}

\newcommand{\cC}{\mathcal{C}}

\newcommand{\cF}{\mathcal{F}}
\newcommand{\cG}{\mathcal{G}}

\newcommand{\cL}{\mathcal{L}}

\newcommand{\cT}{\mathcal{T}}

\newcommand{\cW}{\mathcal{W}}

\newcommand{\NN}{\mathbb{N}}

\newcommand{\RR}{\mathbb{R}}

\def\dsqcup{\sqcup\mathchoice{\mkern-7mu}{\mkern-7mu}{\mkern-3.2mu}{\mkern-3.8mu}\sqcup}

\newcommand{\br}{\mathbf{r}}
\newcommand{\bs}{\mathbf{s}}
\newcommand{\bS}{\mathbf{S}}
\newcommand{\bt}{\mathbf{t}}
\newcommand{\bT}{\mathbf{T}}

\newcommand{\bv}{\mathbf{v}}

\newcommand{\boo}{\mathbf{0}}

\newcommand{\ssi}{\Sigma}
\newcommand{\si}{\sigma}

\newcommand{\lp}{\left(}
\newcommand{\rp}{\right)}
\newcommand{\lc}{\left[}
\newcommand{\rc}{\right]}
\newcommand{\lcl}{\left\{}
\newcommand{\rcl}{\right\}}

\newcommand{\lla}{\left\langle}
\newcommand{\rra}{\right\rangle}

\newcommand{\cdg}{\color{dg}}

\definecolor{dg}{rgb}{0, 0.5, 0}
\definecolor{dp}{rgb}{0.50, 0, 0.40}

\newcommand{\la}{\langle}
\newcommand{\ra}{\rangle}

\newcommand{\sym}{\mathrm{Sym}}

\newcommand{\eps}{\varepsilon}
\newcommand{\dd}{\mathop{}\!\mathrm{d}}

\begin{document}


\title[On the signature of an image]{On the signature of an image}


\author{Joscha Diehl \and Kurusch Ebrahimi-Fard \and Fabian N.~Harang \and Samy Tindel}

\address{Joscha Diehl: Universit\"at Greifswald,
Institut fur Mathematik und Informatik, Walther-Rathenau-Str. 47, 17489 Greifswald }
\email{joscha.diehl@uni-greifswald.de}

\address{Kurusch Ebrahimi-Fard: Department of Mathematical Sciences, Norwegian University of Science and Technology (NTNU), 7491 Trondheim, Norway. Centre for Advanced Study (CAS), 0271 Oslo, Norway.}
\email{kurusch.ebrahimi-fard@ntnu.no}
\urladdr{https://folk.ntnu.no/kurusche/}

\address{Fabian A. Harang: Department of Economics, BI Norwegian Business School, Handelshøyskolen BI, 0442, Oslo, Norway.}
\email{fabian.a.harang@bi.no}

\address{Samy Tindel, Department of Mathematics, 
Purdue University, 150 N. University Street, West Lafayette, IN 47907, USA}
\email{stindel@purdue.edu}
\urladdr{https://www.math.purdue.edu/~stindel/}

\date{\today}


\begin{abstract}
Over the past decade, the importance of the 1D signature which can be seen as a functional defined along a path, has been pivotal in both path-wise stochastic calculus and the analysis of time series data. By considering an image as a two-parameter function that takes values in a $d$-dimensional space, we introduce an extension of the path signature to images. We address numerous challenges associated with this extension and demonstrate that the 2D signature satisfies a version of Chen's relation in addition to a shuffle-type product. Furthermore, we show that specific variations of the 2D signature can be recursively defined, thereby satisfying an integral-type equation. We analyze the properties of the proposed signature, such as continuity, invariance to stretching, translation and rotation of the underlying image. Additionally, we establish that the proposed 2D signature over an image satisfies a universal approximation property.  
\end{abstract}


\thanks{
\emph{Acknowledgments.} 
All authors are grateful to the Center for Advanced Studies (CAS) in Oslo for funding the "Signatures for Images" project over the academic year 2023/2024, during which the writing of this article happened.
JD was supported by a Trilateral ANR-DFG-JST AI program, DFG Project 442590525. }


\maketitle

\tableofcontents
\addtocontents{toc}{~\hfill\textbf{Page}\par}



\section{Introduction}
\label{sec:intro}

To enhance the accuracy of representations in data analysis, it is often important to enrich the original data using nonlinearities. 
In much of the current literature on data processing, neural networks are often trained without explicit guidance on determining the appropriate type of nonlinearity to be added. Nevertheless, it may be more efficient to begin with a natural and easily interpretable notion of nonlinearity within the relevant context.

With this general understanding, the concept of a 1D signature%
\footnote{In this work
1D and 2D will correspond to
the \emph{parameter} space.
All curves or fields (or membranes)
will have arbitrary ambient dimension $d$.
}
defined over a curve has proven to be crucial in both path-wise stochastic calculus and the analysis of time series data. For a given ambient dimension $d \ge 1$ and interval $\ott$, let us briefly recall the definition of the signature of a differentiable path 
\begin{equation*}
x=\lcl
    x_{t}^{i}\ | \, i=1,\ldots,d, \, t\in\ott
\rcl.
\end{equation*}
To achieve this, let us initially define the set of words on a finite alphabet.

\newcommand\e{\mathsf{e}}
\begin{defn}
\label{def:words}
Let $\cA=\{1,\ldots, d\}$ be an alphabet. Denoting the collection of words with $n$ letters
\begin{equation}
\label{a0}
\cW_{n} :=
\lcl
w=(i_{1},\ldots,i_{n})\ | \, 
n\ge 0 \text{ and } i_{j}\in\cA \text{ for all } j=1,\ldots, n
\rcl,
\end{equation}
where the case $n=0$ corresponds to the empty word $\e$, we define the set of all words over $\cA$ as $\cW := \cup_{n\ge 0} \cW_{n}$ ($\cW' := \cW - \{\e\}$). The length of a word $w=(i_{1},\ldots,i_{m})$ is denoted $|w|=m$.
\end{defn}

\noindent
Now one can define the signature $S(x)$ as an infinite dimensional object in $(\cC(\Delta_{0,T}^{2}))^{\cW}$, where the simplex
$$
\Delta_{0,T}^{2}:=\{(s,t)\ | \,0\le s \le t \le T\},
$$
by specifying recursively the projections $\la  S(x), \, w \ra$ for every word $w\in\cW$. Namely, we set $\lla S_{st}(x) , \, \e  \rra :=1$ and if  $w=(i_{1},\ldots,i_{n})\in\cW'$, then we define
\begin{equation}
\label{a1}
\lla S_{st}(x) , \, w  \rra
:=
\int_{s}^{t} \lla S_{sr}(x) , \, (i_{1},\ldots,i_{n-1})  \rra \, \dd x_{r}^{i_{n}} \, .
\end{equation}

The rich structure enjoyed by $S(x)$ has first been emphasized by the mathematician K.~T. Chen~\cite{Ch} within an algebraic context. Subsequently, the concept of signature was imported by T.~Lyons \cite{Ly} to the realm of stochastic analysis, leading to the development of the fruitful theory of rough paths. Among the properties which have made the notion of path signature so insightful, let us mention the following ones:

\begin{enumerate}[wide, labelwidth=!, labelindent=0pt, label= \textbf{(\roman*)}]
\setlength\itemsep{.05in}

\item
Algebraic properties like Chen's relation as well as shuffle products. This yields a wealth of structures, such as Lie groups and algebras, as well as Hopf algebras, providing a convenient mathematical framework for studying the path signature $S(x)$.

\item
Regularity properties in terms of H\"older continuity and $p$-variations.

\item
Ability to expand solutions of differential equations in terms of $S(x)$, leading to pathwise analysis of noisy equations.

\item
The fact that a generic curve can be recovered from its signature, up to tree-like equivalence.

\end{enumerate}

\noindent
We refer e.g.~to \cite{FH,FV,LCL} for more details about those properties. In any case, items (i)-(iv) above have been a sufficient motivation to consider signatures as proper features for classification tasks related to curves. See e.g.~references \cite{ChevyKrom2016,LyonsMcL2022}.  Notice the special importance of (iv) above in this regard.  In the past decade, there has been a significant upswing in the importance of utilizing signature methods in diverse data analysis applications. This trend extends across various domains, encompassing Chinese character recognition \cite{Graham2013}, topological data analysis \cite{ChevNanObe2018}, and the development of signature-based machine learning models for psychiatric diagnosis \cite{PerGoGedLySa2018,WuGoodLySau2022}. 

Based on those preliminary remarks, it is rather natural to ask if appropriate generalizations of the notion of signature can shed new light on the analysis of images. Indeed, even if nonlinear features have been a prominent fixture of image processing in the recent past (see e.g.~\cite{SM}), those features are often defined in a somewhat arbitrary fashion. Moreover, one may remark that a precise and transparent algebraic structure, which is useful for basic analysis as well as efficient implementation of those objects, is also lacking in~\cite{SM} and related papers. Hence, there exists a compelling need for an appropriate extension of the signature concept applicable to 2D-indexed $\mathbb{R}^d$-valued fields. Specifically, our objective is to emulate at least some of the desirable properties (i)-(iv) mentioned earlier, bridging the gap in the current literature and fostering a more comprehensive understanding of image analysis.

We briefly mention earlier work where the concept of 2D signature is touched. In \cite{horozov2015noncommutative}, Horozov explores a generalization of Chen's iterated integrals from paths to membranes, i.e., 2D surfaces or fields in our parlance. He describes a shuffle identity which is similar to ours. However, we note two major differences. First,  Horozov integrates first-order differentials. Second, he works in ambient dimension $2$ which implies homotopy invariance, something that is not true in our more general setting. 
Lee et al. \cite{giusti2022topological,lee2023random} pursue
an abstract algebraic-geometric approach, which in the latter reference 
is built around a categorical generalization of Chen's identity to surfaces. In contrast to their works, our proposed signature is constructed from mixed partial differentials of the underlying field, as motivated from multiplicative noisy controls seen in the field of SPDEs. This approach creates different properties of the resulting signature, which will be discussed in depth in the next paragraph. 
In \cite{diehl2022two}, a two-parameter version
of the iterated-sums signature \cite{diehl2020time,diehl2022tropical}
is elaborated.
Finally, in reference \cite{ibrahim2022imagesig} images are considered
as time-series, by going through it ``line by line''.

To extend the path signature \eqref{a1} to surfaces or fields indexed by pairs of parameters from the square $\ott \times \ott$, an initial approach that naturally comes to mind is to employ concepts from rough pathwise calculus within the two-dimensional plane (as introduced in~\cite{CG,CT}). In reference~\cite{ZLT}, the fourth author together with collaborators discretized part of the two-parameter analog of the 1D signature considered in~\cite{CG,CT}. A low (that is $15$-)dimensional set of features was obtained, with good classification performances for a standard set of textures. This preliminary result provided sufficient encouragement to delve deeper into that particular direction.

\smallskip

In the broader context outlined, the present contribution seeks to establish a conceptually straightforward definition of the signature for 2D-indexed random fields. A remark is in order regarding item (iii) above. Our point of view is that the full power of rough calculus in the plane (as exhibited in~\cite{CG,CT}) might entail a level of generalization of the signature that could prove overly intricate for the specific objectives of data analysis. Therefore, referring to items (i)-(iv) above, we shall de-emphasize item (iii) in favor of prioritizing items (i)-(ii) and (iv). We assert that this emphasis adjustment is effectively realized through the consideration of the various 2D signatures described below. Indeed, we introduce three candidates for 2D signatures over a field $X=X(s,t)$: 2D $\id$-signature, $\bS^{\id}(X)$, the full 2D signature, $\bS(X)$, and the symmetrised full 2D signature, $\bS^{\text{sym}}(X)$.

Let us conclude the introduction with a motivation for the coming analysis through a brief overview of some of the main properties that we will prove to hold for the proposed full signature $\bS(X)$:
\begin{enumerate}[wide, labelwidth=!, labelindent=0pt, label= \textbf{(\roman*)}]
\setlength\itemsep{.05in}

    \item The full 2D signature $\bS(X)$ over a field $X:[0,T]^2 \rightarrow \RR^d $ is invariant to translation by constants, but also by 1-parameter paths! This shows that our proposed signature can really be seen as a complementary feature to the classical path signature, and that the 2D signature indeed captures multi-directional relative change in the signal.
    
    \item The full 2D signature of 90-degree rotations of a field $X$ can be expressed in the  2D signature of $X$ itself (up to multiplication by $-1$). So, while we cannot state that the  2D signature is rotation invariant, all necessary information related to 90-degree rotations is contained in the signature.

    \item
    Already in the case of one-dimensional ambient space of the field,
    the proposed full 2D signature is non-trivial
    (i.e.~not equivalent to polynomials in the increment). 
    This is in contrast to the path signature
    as well as the 
    two-parameter
    signatures presented in 
    \cite{giusti2022topological,lee2023random}.

    \item  The full 2D signature is a continuous functional.

    \item Certain terms in the full 2D signature can be used in expressing solutions to certain hyperbolic PDEs, known as Goursat equations,  with multiplicative noise. 

    \item The full 2D signature satisfies a variant of the well-known shuffle relation, and thus, the linear span of 2D signature elements form an algebra.  

    \item The full 2D signature is universal, in the sense that linear combinations of signature terms approximate continuous functionals on the space of $C^2$ fields arbitrarily well. 

    \item The full 2D signature satisfies a Chen-type relation, where a convolution product is involved. Through certain symmetrization arguments in the integrand, similar to what is used in the Wiener--It\^o chaos expansion, one obtains multiplicative Chen-type relations in the horizontal or vertical direction of an image. 
\end{enumerate}

While our work provides new results related to the construction of 2D signature, in addition to a detailed overview of the challenges involved in this extension, we also provide discussions of several open problems and challenges associated with this structure. 

\medskip

The paper is organised in the following Sections:
\begin{itemize}
    
    \item[\ref{sec:prelim}]  We recall some of the basic structures of the path signature.
    
    \item[\ref{sec:planecalc}] We develop notation, discuss multiparameter integral operators, and provide an overview of calculus in the plane.
    
    \item[\ref{sec:2dsignature}]
    We propose a 2D signature over a two-parameter, $\RR^d$-valued field.
    In fact, starting from the expansion of certain PDEs, we obtain
    an object (the $\id$-signature). Chen's relation is discussed in this  2D setting. However, we will show that the $\id$ signature lacks some of the desired
    properties. 
    This motivates the definition of a full 2D signature,
    defined over a larger word shuffle algebra. 
    Computations are presented to exemplify the properties of the full 2D signature.

    \item[\ref{sec:symetrizedSig}]  We symmetrize the integrand, and observe that this new signature type feature satisfies Chen's relation.
    
    \item[\ref{sec:Funcapprox}] A proof is given, that the 2D signature characterizes the field.
    
    \item[\ref{sec:conclussion}] We provide a few concluding remarks and discuss open problems and future progress.
    
\end{itemize}


\section{Linear ODEs and the 1D signature}
\label{sec:prelim}

In this section, we recall how signatures for curves parametrized over an interval $\ott$ can be motivated in terms of linear ordinary differential equations.
We first show how the signature $S(x)$ in~\eqref{a1} emerges from $\RR^{e}$-valued differential equations.
Then we identify the signature itself as the solution of a linear equation in an infinite dimensional space. Those notions will be generalized later to a multi-parametric setting.


\subsection{Motivation from real-valued differential equations}
\label{ssec:motivation}

Signatures of curves parametri\-zed over $\ott$ appear naturally in the context of linear fixed point equations (e.g., see~\cite{HT} for this perspective in the context of Volterra equation).

\begin{lem}
\label{lem:linear-eq-1d}
Let $x:\ott\to\RR^{d}$ be a $\cC^{1}$ signal and let $\{A^{i}\ | \, i=1,\ldots,d\} \subset \RR^{e\times e}$ be a collection of constant matrices with real entries. We consider the $\RR^{e}$-valued solution $y$ to the linear equation 
\begin{equation}\label{c1}
y_{st}:=y_t-y_s
=
\sum_{i=1}^{d} \int_{s}^{t} A^{i}  y_{r}\, \dd x_{r}^{i},
\end{equation}
for $(s,t)\in\Delta^{2}_{0,T}$, with a given initial condition $y_s \in\RR^{d}$. Recalling \eqref{a0}, we set
\begin{equation}\label{c11}
A^{\circ w}
:=
A^{i_{n}} \cdots A^{i_{1}},
\quad\text{for}\quad
w=(i_{1},\ldots,i_{n}) \in \cW',
\end{equation}
and $A^\e$ is the $e \times e$ identity matrix $\mathbf{1}_e$.
Then equation \eqref{c1} can be expanded as
\begin{equation}
\label{eq:linear-eq-series-expansion}
y_{t}
=
\sum_{w\in\cW} A^{\circ {w}} y_{s} \lla S_{st}(x) ,  w  \rra ,
\end{equation}
where the quantity $\la S_{st}(x) ,  w  \ra$ was introduced in \eqref{a1}.
\end{lem}

\begin{proof}
We will provide a brief outline of the proof here,  establishing the foundation for the forthcoming exploration of the two-parameter case. Employing Einstein's summation convention for repeated upper indices, equation~\eqref{c1} transforms into:
\begin{equation*}
y_{st}
=
\int_{s}^{t} A^{i_{1}}  y_{r_{1}}\, \dd x_{r_{1}}^{i_{1}} \, ,
\quad\text{for}\quad (s,t)\in\Delta^{2}_{0,T} .
\end{equation*}
Replacing $y_{r_{1}}$ by $y_{sr_{1}} + y_{s}$ on the righthand side yields
\begin{equation}\label{c2}
y_{st}
=
A^{i_{1}} y_{s} \int_{s}^{t} \dd x_{r_{1}}^{i_{1}} + R_{st}^{1} \, ,
\quad\text{where}\quad
R_{st}^{1}
:=
\int_{s}^{t} A^{i_{1}}y_{sr_{1}}  \, \dd x_{r_{1}}^{i_{1}} .
 \end{equation}
Repeating this step, the increment $y_{sr_{1}}$ in the remainder term $R_{st}^{1}$ can be further expanded
\begin{equation*}
y_{sr_{1}}
=
A^{i_{2}}y_{s}  \int_{s}^{r_1} \dd x_{r_{2}}^{i_{2}} 
+
\int_{s}^{r_1} A^{i_{2}} y_{sr_{2}} \, \dd x_{r_{2}}^{i_{2}} \, .
\end{equation*}
Plugging this into~\eqref{c2}, we end up with
\begin{equation}\label{c3}
y_{st}
=
 A^{i_{1}}y_{s} \int_{s}^{t} \dd x_{r_{1}}^{i_{1}} 
+
A^{i_{1}} A^{i_{2}}y_{s} \int_{s\le r_{2}< r_{1} \le t} \dd x_{r_{2}}^{i_{2}} \dd x_{r_{1}}^{i_{1}}
+
R_{st}^{2}.
\end{equation}
The expression of the remainder $R_{st}^{2}$ follows the same procedure. In addition, observe that invoking our conventions \eqref{a1} and \eqref{c11}, one can recast~\eqref{c3} as
\begin{equation*}
y_{st}
=
A^{\circ (i_{1})} y_{s} \lla S_{st}(x) ,  (i_{1})  \rra
+
A^{\circ (i_{2},i_{1})} y_{s} \lla S_{st}(x) ,  (i_{2},i_{1})  \rra
+
R_{st}^{2} .
\end{equation*}
Our claim~\eqref{eq:linear-eq-series-expansion} follows
from iterating this procedure.
\end{proof}


\subsection{The algebra of signatures}
\label{ssec:Sigalg}

Following \cite{FV}, we start this subsection by recalling the algebraic and geometric setting for 1D signatures indexed by a single parameter. This will allow us to provide a linear differential equation which governs the signature of a path. We will try to replicate this linear equation for images in the next sections. 

We begin by recalling the definition of tensor algebra over the space $\mathbb{R}^{d}$ 
\begin{equation*}
\cT(\mathbb{R}^{d}):=\bigoplus_{n=0}^{\infty}(\mathbb{R}^{d})^{\otimes n},
\end{equation*}
with $(\mathbb{R}^{d})^{\otimes
0}=\mathbb{R}\mathbf{1}$. It is equipped with a product $\otimes$ defined for any $g,h\in
\cT(\mathbb{R}^{d})$ by
\begin{equation}
\label{tensorprod}
\lc g\otimes h\rc^{n}=\sum_{k=0}^{n}[g]^{n-k}\otimes [h]^{k},
\end{equation}
where $[g]^{n} \in (\mathbb{R}^{d})^{\otimes n}$ designates the projection onto the $n$-th tensor level, making $\cT(\mathbb{R}^{d})$ an associative and non-commutative algebra with unit $\mathbf{1} \in (\mathbb{R}^{d})^{\otimes 0}$. 

The space $\cT(\mathbb{R}^{d})$ can be equipped with the commutative shuffle product defined inductively, i.e., $w \shuffle \mathbf{1} = \mathbf{1} \shuffle w = w$ for any $w \in \cT(\mathbb{R}^{d})$ and for $u \otimes w_1$, $v \otimes w_2$, $u,v \in (\mathbb{R}^{d})^{\otimes 1}$, $w_1,w_2 \in \cT(\mathbb{R}^{d})$, we define
\begin{equation}\label{eq:shuffle}
        (u \otimes w_1) \shuffle (v \otimes w_2)
    := u \otimes (w_1 \shuffle (v \otimes w_2))
    +  v \otimes ((u \otimes w_1) \shuffle w_2).
\end{equation}
For example, shuffling two letters, $u$ and $v$, gives 
$$
    u \shuffle v = u \otimes v + v \otimes u.
$$
See \Cref{def:shuffle} 
for an explicit formulation of shuffle product in terms of permutations. Further below (see Remark \ref{not:shuffle-on-words}), we will use word notation to denote elements in $\cT(\mathbb{R}^{d})$, i.e., if $e_{1},\ldots,e_{d}$ denotes the canonical basis of $\RR^{d}$, then we use $(i_{1},\ldots,i_{n})\in\cW$ to denote $e_{i_{1}}\otimes\cdots\otimes e_{i_{n}} \in \cT(\mathbb{R}^{d})$.%

The space of tensor series is denoted $\cT((\mathbb{R}^{d}))$. It contains linear maps $F:=\sum_{w \in \cW}f_w e_w$ from $\cT(\mathbb{R}^{d})$ to $\mathbb{R}$, where 
\begin{equation}
\lcl
e_{w}=e_{i_{1}}\otimes\cdots\otimes e_{i_{n}} \ ; \, w=(i_{1},\ldots,i_{n})\in\cW
\rcl
\end{equation} 
such that for $v \in \cW$, we define $F(v)=\sum_{w \in \cW}f_w \la e_{w} , v \ra =f_v$. The tensor product \eqref{tensorprod} can be extended to $\cT((\mathbb{R}^{d}))$ making it an unital algebra, with unit  $\mathbf{1}$ (in place for the empty word).

A continuous map $\Phi:\Delta_{a,b}^{2}\rightarrow \cT((\mathbb{R}^{d}))$ is called a multiplicative functional if 
$$
\Phi_{st} =\Phi_{su}\otimes\Phi_{ut}
,$$ 
for $s<u<t$. The signature $S(x)$ of a path $x$ with finite variation is a particular example of such a multiplicative functional. Indeed, recalling~\eqref{a1} we can embed $\la S(x), w \ra$ by setting
\begin{equation}
\label{c4}
S_{st}(x) = \mathbf{1} + \sum_{w\in\cW'} \la S_{st}(x), w \ra \, e_{w} \in \cT((\RR^{d})).
\end{equation}
In addition, it can be shown that $S_{st}(x)$ lives in a subset $G(\mathbb{R}^{d}) \subset \cT((\mathbb{R}^{d}))$ consisting of so-called group-like elements
\begin{equation*}
\mathcal{G}(\mathbb{R}^{d}) = \exp^\otimes\bigl(\mathcal{L}((\mathbb{R}^d))\bigr),
\end{equation*}
where $\mathcal{L}((\mathbb{R}^d))$ denotes the set of all Lie series over $\mathbb{R}^{d}$. Hence, every element in $\mathcal{G}(\mathbb{R}^{d})$ is the tensor-exponential of a Lie series.
Regarding the latter, recall that, denoting by $\mathcal{L}(\mathbb{R}^{d})$ the free Lie algebra generated by $\mathbb{R}^{d}$, an element in $\cT((\mathbb{R}^{d}))$ is a Lie series if it can be written in the form $\sum_{n \geq 1} p_n$, where each Lie polynomial $p_n \in \mathcal{L}(\mathbb{R}^{d})$ has support in $\cW_n$.   

As seen above, the notion of signature can be motivated by linear differential equations. We now state a differential/integral equation in $\mathcal{G}(\mathbb{R}^{d})$ for $S_{st}(x)$ itself. It should be seen as a simple non-truncated version of \cite[Prop.~7.8]{FV}.

\begin{prop}\label{prop:path sig eq}
For $(s,t) \in \Delta_{a,b}^{2}$, the signature $S_{st}(x)$ defined by \eqref{c4} solves a linear integral fixed point equation in $\cT((\RR^d))$ of the form
\begin{equation}\label{eq:path Sig diff}
S_{st}(x)
= \mathbf{1}+
 \int_{s}^{t} S_{sr}(x) \otimes  \dd x_{r}.
\end{equation}
\end{prop}


\section{Simplexes and integration}
\label{sec:planecalc}

This section is devoted to collecting some notations as well as elementary results used across the article in the context of calculus in the plane. 


\subsection{Rectangles and simplexes}
\label{ssec:notation}

Throughout the text, rectangles in the square $\ott^{2} \subset \mathbb{R}^2$ will play a crucial role. We first spell out our convention and notation for those objects. 

\noindent
{\bf{Convention}}: we refer to the horizontal and vertical axis in the plane as the first respectively second variable, indexed respectively by subscripts $1$ and $2$. See Figure~\ref{fig:generic-rectangle}.

\begin{nota}\label{not:rectangles}
A generic rectangle $R$ in $\mathbb{R}^2$ is of the form $R=[s_1,t_1]\times [s_2,t_2]$, with $s_{1}\le t_{1}$ and $s_{2}\le t_{2}$. A more compact notation follows by considering points in $\mathbb{R}^2$, setting $\mathbf{s}=(s_{1},s_{2}) $ and $\mathbf{t}=(t_{1},t_{2})$, where $s_1 \le t_1 \le T$, $s_2 \le t_2 \le T$. We then denote by
$R=[\mathbf{s},\mathbf{t}]:= [s_1,t_1]\times [s_2,t_2]$  a sub-rectangle specified by its lower left and upper right
corners, inside the standard-square $[0,T]^2$.  
We shall also write $\bT=(T,T)$ for any positive quantity $T$.
\end{nota}

\begin{nota}
\label{def:shuffle}
Denote by $\ssi_{\mathcal{A}}$ the set of permutations of elements of a finite set $\mathcal{A}$. 
Furthermore, for $n,k\geq 1$ we consider the following set of permutations of $\{1,\ldots,n+k\}$:
\begin{equation}\label{eq:shuffle-permut}
\sh(n,k)
=
\lcl
\rho\in\ssi_{\{1,\ldots,n+k\}}\ | \,
\rho(i) < \rho(j) \text{ if } 1 \le i < j \le n \text{ or } n+1\le i < j \le n+k
\rcl.
\end{equation}

\end{nota}
\begin{center}
\begin{figure}[htbp]
\caption{Typical example of rectangle $[\mathbf{s},\mathbf{t}]$ in $\ott^{2}$.}
\label{fig:generic-rectangle}
\begin{tikzpicture}[xscale=0.4,yscale=0.4]


\draw [->,  very thick] (-2,0) --(23,0) node[below,scale=.8] {{\cdg Axis 1}};
\draw [->,  very thick] (0,-2) --(0,15) node[right,scale=.8] {{\cdg Axis 2}};

\draw [fill,black] (2,0) circle [radius=.2];
\node [below,scale=.8, blue] at (2,0) {$s_{1}$}; 

\draw [fill,black] (20,0) circle [radius=.2];
\node [below,scale=.8, blue] at (20,0) {$t_{1}$}; 

\draw [fill,black] (0,3) circle [radius=.2];
\node [left,scale=.8, blue] at (0,3) {$s_{2}$}; 

\draw [fill,black] (0,12) circle [radius=.2];
\node [left,scale=.8, blue] at (0,12) {$t_{2}$}; 

\draw[blue, very thick] (2,3) rectangle (20,12);

\draw [fill,red] (2,3) circle [radius=.15];
\node [below,scale=.8, red] at (3,3) {$\mathbf{s}=(s_{1},s_{2})$};
\draw [fill,blue] (20,3) circle [radius=.15];
\draw [fill,red] (20,12) circle [radius=.15];
\node [above,scale=.8, red] at (20,12) {$\mathbf{t}=(t_{1},t_{2})$};
\draw [fill,blue] (2,12) circle [radius=.15];
\end{tikzpicture}
\end{figure}
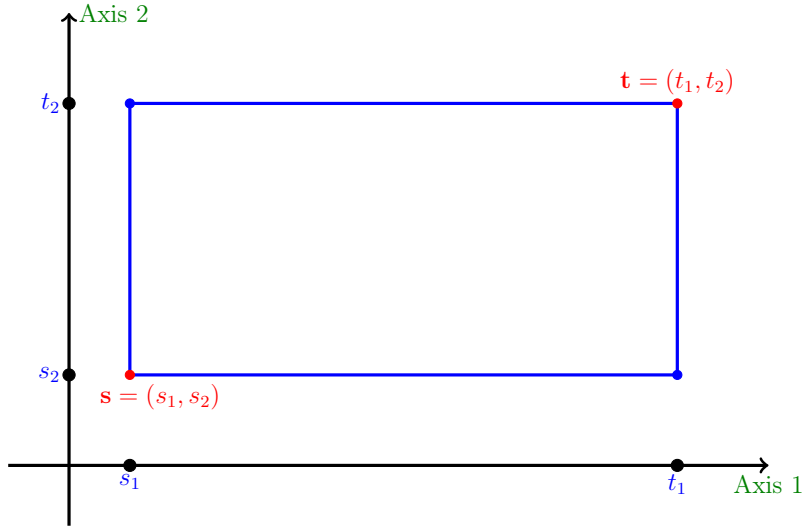
\end{center}
Our upcoming computations employ the concept of rectangular increment, and will heavily depend on simplexes, including a variant that incorporates permutations. 

\begin{defn}
\label{def:simplexes}
Let $a \leq b$ be elements in $\ott$. Then we define the following simplexes:
\begin{enumerate}[wide, labelwidth=!, labelindent=0pt, label= \textbf{(\roman*)}]
\setlength\itemsep{.05in}

\item 
The usual simplex of order $n$ is given by
\begin{equation}
\label{standsimp}
\Delta_{a,b}^n
=
\big\{ r\in [0,T]^n \, | \, a\leq r^{1}\leq \cdots\leq r^{n}\leq b\big\}.
\end{equation}

 \item
 Consider a permutation $\si \in \ssi_{\{1,\ldots,n\}}$. Then we define the $\si$-simplex
 \begin{equation}
 \label{b2}
\Delta^n_{a,b} \sigma = \big\{ r\in [0,T]^n \, | \, a\leq r^{\sigma(1)}\leq \cdots\leq r^{\sigma(n)}\leq b\big\}.
\end{equation}
\end{enumerate}  
\end{defn}

\begin{rem}
Note the use of superscripts in \eqref{standsimp} and \eqref{b2} to denote components of $r\in [0,T]^n$. 
\end{rem}


It is well-known that permutations carry a shuffle property, which can be translated into a corresponding property for simplexes. We recall this elementary result for further use.
For a proof of this statement, we refer the reader to Patras  \cite[Prop.~2]{Patras1991}.
\begin{lem}
    \label{lem:prod of simplex} 
    Consider two positive integers $n,k$, and let $\si\in\Sigma_{\{1,\ldots,n\}}$ and $\tau\in\Sigma_{\{1,\ldots,k\}}$.
For $a\le b$ we define $\Delta_{a,b}^{n}\si$ and $\Delta_{a,b}^{k}\tau$ like in~\eqref{b2}. Then the product $\Delta_{a,b}^{n}\si \times \Delta_{a,b}^{k}\tau$ can be ``linearized'':
    \begin{equation}\label{eq:prod of delta}
    \Delta_{a,b}^{n}\si \times \Delta_{a,b}^{k}\tau
    =
    \bigcup_{\rho\in\sh(n,k)} \Delta_{a,b}^{n+k} (\sigma * \tau) \circ \rho^{-1} \, ,
    \end{equation}
    where $\sigma *\tau \in \Sigma_{\{1,\ldots,n+k\}}$
    is defined as
    \begin{align*}
       (\sigma * \tau)(i)
       :=
       \begin{cases}
           \sigma(i) & i \le n \\
           \tau(i-n) + n & n + 1 \le i.
       \end{cases}
    \end{align*}
    The union is disjoint up to sets of zero Lebesgue measure.
\end{lem}

We close this section by defining the main type of simplex which will be used in the context of integration in the plane.

\begin{nota}\label{not:2d-differentials}
Let $\bs=(s_{1},s_{2})$ and $\bt=(t_{1},t_{2})$ be two elements of $\ott^{2}$, where we are using Notation~\ref{not:rectangles} on rectangles. Consider an integer $n\ge 2$. We denote the set 
\begin{multline}\label{eq:def-2d-simplex}
\Delta_{[\bs,\bt]}^{n}
:=
\lcl
(\br^{1},\ldots,\br^{n})\in\lp\ott^{2}\rp^{n} ; \,
s_{1}\le r_{1}^{1} \le \cdots\le r_{1}^{n} \le t_{1} \text{ and } s_{2} \le r_{2}^{1} \le\cdots \le r_{2}^{n}\le t_{2}
\rcl.
\end{multline}
\end{nota}


\subsection{Some properties of integral operators}

To facilitate future computations associated with the construction and properties of signature over an image, this section will revisit several essential properties of iterated integral operators. Commencing with a straightforward definition, we clarify the concept of an iterated integral operator:

\begin{defn}\label{def:intg-operator}
Let $n$ be a positive integer, and consider a measurable set $A\subset [0,T]^{n}$. We write $\cL(V;W)$ for the set of bounded linear operators from a vector space $V$ to a vector space $W$. Then we define an integral operator $\int_{A}$ as an element in $\cL(L^{1}([0,T]^{n}); \RR)$ such that
\begin{equation}
\label{b31}
 \int_{A}: \      X\in L^{1}([0,T]^{n}) \longmapsto 
  \int_{A}X:= \idotsint\limits_{A} X(r_1,\ldots,r_n) \, \dd r_1\cdots \dd r_n  \in \RR.
    \end{equation}
\end{defn}    
    
\noindent
In the sequel, we shall compose integral operators and therefore introduce a particular notation for that purpose.

\newcommand\starI{\star'} 
\newcommand\interlace{\mathbin{\rotatebox[origin=c]{180}{\(\shuffle\)}}}

\begin{defn}
\label{not:star}
We define two types of  compositions of integral operatos as follows: 
\begin{enumerate}[wide, labelwidth=!, labelindent=0pt, label= \textbf{(\roman*)}]
\setlength\itemsep{.05in}
\item Let $A\subset [0,T]^m$ and $B\subset [0,T]^n$ be two measurable sets. Then we define 
\begin{equation}\label{b4}
    \int_A \, \star \, \int_B \equiv \int_{A\times B} \ ,
\end{equation}
where the right hand side of \eqref{b4} is understood as in Definition~\ref{def:intg-operator}.

\item For two measurable sets  $A,B \subset \RR^n$, we define 
\begin{equation}\label{interlace sets}
    \int_A \, \starI \, \int_B \equiv
    \int_{ A \interlace B } \, ,
\end{equation}
where
the ``interlacing''
of $A$ and $B$ is defined as the set
\begin{align*}
    A \interlace B
    :=
    \{ ( r_1^1,r_1^2, \dots, r_1^n, r_2^n) \in \RR^{2n}
    \ \mid\ 
    (r_1^1, \dots, r_1^n) \in A,
    (r_2^1, \dots, r_2^n) \in B \}.
\end{align*}
\end{enumerate}
Note that, if $A,B\subset \RR$, then
$\int_A \star \int_B =  \int_A \starI \int_B$.
\end{defn}

The following ``interchange law'' is immediate.
\begin{lem}
    \label{lem:interchange}
    For $A, B \subset \RR^m, A',B' \subset \RR^n$,
   \begin{align*}
       \lp \int_A \star \int_{A'} \rp
       \starI
       \lp \int_B \star \int_{B'} \rp
       =
       \lp \int_A \starI \int_B\rp 
       \star
       \lp\int_{A'} \starI \int_{B'}\rp.
   \end{align*} 
\end{lem}
 
\begin{example}
As an illustration of the previous notation, let us consider the particular case of simplexes. Here we use Definition~\ref{def:simplexes} and Notation~\ref{not:2d-differentials}. Consider now $(s_1,t_1)$ and $(s_2,t_2)$ in $\Delta_{0,T}^2$, where we recall the notation given in \eqref{standsimp} for $\Delta_{a,b}^n$. Then 
\begin{equation*}
\int_{\Delta^n_{s_1,t_1}} \star \int_{\Delta^n_{s_2,t_2}}
=
\int_{\Delta^n_{s_1,t_1}\times \Delta^n_{s_2,t_2}}
\end{equation*}
 and
\begin{align*}
 \int_{\Delta^n_{s_1,t_1}} \starI \int_{\Delta^n_{s_2,t_2}}
 =
 \int_{\Delta_{[\bs,\bt]}^n}.
\end{align*}
\end{example}
    
The main reason why we have introduced the notion of integral operator is the following: we wish to highlight the fact that the usual algebraic properties of iterated integrals can be reduced to manipulations on simplexes. Let us first rephrase Chen's relation in this framework.

\begin{lem}
    \label{lem:Chens integral relation}
    Let $T>0$ and consider $(s,u,t)\in \Delta^3_{0,T}$. Recall our notation~\eqref{b31} for integral operators and relation~\eqref{b4} for compositions. Then we have that 
    \begin{equation}\label{b5}
        \int_{\Delta_{s,t}^n }
        =\sum_{l+k=n} \int_{\Delta_{s,u}^l}\star \int_{\Delta^k_{u,t}}.
    \end{equation}
\end{lem}
\begin{proof}
    The following fact about simplexes is well-known and easy to prove: 
    \begin{equation}\label{b6}
        \Delta_{s,t}^n =\bigcup_{l+k=n} \Delta_{s,u}^l\times \Delta^k_{u,t},
    \end{equation}
    where the union is disjoint, up to a set of measure zero,
    and where it is understood that $\Delta_{s,u}^{n}\times \Delta^{0}_{u,t}=\Delta_{s,u}^{n}=\Delta^{0}_{u,t} \times \Delta_{s,u}^{n}$.
    Our claim~\eqref{b5} then follows.
\end{proof}

\begin{rem}
    As a particular case of Lemma \ref{lem:Chens integral relation}, we will later use the following relation for $(s,u,t)\in \Delta^3_{0,T}$: 
   \begin{equation}
   \label{eq:special chen 2 order}
    \int_{\Delta^1_{s,u}}\star \int_{\Delta^1_{u,t} }= \int_{\Delta_{s,t}^2}-\int_{\Delta_{s,u}^2}-\int_{\Delta_{u,t}^2}. 
   \end{equation}
\end{rem}

\begin{rem}\label{rem:algebraic integration}
We can recast relation \eqref{b5} in the language of algebraic integration as described in~\cite{GUBINELLI200486}. To this aim we define the space $C=C_{1}$ of continuous functions $x:[0,T]\rightarrow \RR^d$. We also set $C_2$ as the space of continuous functions $y:[0,T]^2\rightarrow \RR^d$  vanishing on the diagonal (i.e.~$y_{ss}=0$). Next we define an operation $\delta:C_i\rightarrow C_{i+1}$ by 
\begin{equation}
\label{eq:delta def}
    \delta_t x_s =x_t-x_s,\quad\text{and}\quad \delta_u y_{s,t}= y_{s,t}-y_{s,u}-y_{u,t}.
\end{equation}
Then Lemma \ref{lem:Chens integral relation} asserts that for $(s,u,t)\in \Delta^3_{0,T}$
\begin{equation}
\label{eq:ff}
    \delta_u \int_{\Delta^n_{s,t}} \equiv \int_{\Delta^n_{s,t}}-\int_{\Delta^n_{s,u}}-\int_{\Delta^n_{u,t}}
    =\sum_{l+k=n \atop k,l >0} \int_{\Delta^l_{s,u}}\star \int_{\Delta^k_{u,t}}. 
\end{equation}    
\end{rem}

As mentioned in the introduction, iterated integral operators satisfy the shuffle property. This relation is rather useful in both the analysis and applications of the path signature. We write it here purely in terms of the integral operator.  

\begin{prop}[Shuffle relation of permuted integral operators]\label{prop:shuffle operators}
Recall the notion of permuted simplex, $\Delta_{s,t}^n\rho$, introduced in~\eqref{b2}.
For $s\leq t\in [0,T]$ and permutations $\rho \in \ssi_{\{1,\ldots,n\}}$ and $\nu \in \ssi_{\{1,\ldots,m\}}$, we have that 
\begin{equation}
\label{b7}
    \int_{\Delta_{s,t}^n\rho}\star \int_{ \Delta_{s,t}^{m}\nu  }
=\sum_{\si \in \sh(n,m)} \int_{ \Delta_{s,t}^{n+m}(\rho * \nu)\circ\si^{-1}},
\end{equation}
where $\sh(n,m)$ is given in \Cref{def:shuffle}
and the product $*$ on permutations is defined in \Cref{lem:prod of simplex}.
\end{prop}

\begin{proof}
Owing to \eqref{b4}, we have
\begin{equation*}
\int_{\Delta_{s,t}^n\rho}\star \int_{ \Delta_{s,t}^{m}\nu  }
 =
 \int_{\Delta_{s,t}^n\rho \times \Delta_{s,t}^{m}\nu}.
\end{equation*}
Our claim \eqref{b7} is then a direct consequence of Lemma \ref{lem:prod of simplex}. 
\end{proof}

\noindent

We conclude this section with a useful symmetrization property relating iterated integral operators to integral operators over hypercubes. 

\begin{lem}
\label{lem:simplex to hypercubes}
Let $(s,t)\in\Delta_{0,T}^{2}$.
Then
  \[
  \sum_{\rho\in \ssi_{\{1,\ldots,n\}}}\int_{\Delta^n_{s,t}\rho} = \int_{[s,t]^n}.
  \]
\end{lem}

\begin{proof}
This is known to result from integration by parts. In geometric terms, it follows from the fact that the hypercube $[s,t]^n$ can be divided into $n!$ different simplexes defined by all possible ways of ordering $n$ variables from $[s,t]^n$. This is described by all possible permutations of the set $\{1,\ldots,n\}$, and thus we have 
\[
    \bigcup_{\rho \in \ssi_{\{1,\ldots,n\}}}\Delta_{s,t}^n\rho = [s,t]^n.
\]
\end{proof}


\subsection{Increments in the plane}
\label{sec:calculus-plane-2}

In the sequel, a field $X$ is indexed by a pair of elements from the standard-square $\ott^{2}$
(analogous to a curve $x$ being indexed by elements from $\ott$). We first introduce notation for the so-called rectangular increment of a field $X$. 

\begin{nota}\label{not:2d-deriv}
Let $X \in \cC^{2}(\ott^{2};\RR^{d})$ be a field. For $\br=(r_{1},r_{2})$ we set
\begin{equation}
\partial_{1}X_{\br} = \frac{\partial X_{\br}}{\partial r_{1}},
\qquad
\partial_{2}X_{\br} = \frac{\partial X_{\br}}{\partial r_{2}},
\qquad\text{and}\qquad
\partial_{12}X_{\br} = \frac{\partial^2 X_{\br}}{\partial r_{1} \partial r_{2}} .
\end{equation}
For a function $f:\RR^{d}\to\RR$, we also use the standard notation $\partial_{i}f(x)=\frac{\partial f}{\partial x^{i}}(x)$.
\end{nota}

\begin{defn}
\label{def:squareincrement}
Consider an $\RR^d$-valued function $X:[0,T]^2\rightarrow \RR^d$, and a generic rectangle $R=[\mathbf{s},\mathbf{t}]$, specified by its lower left and upper right corners, $\mathbf{s}=(s_1, s_2)$ respectively $\mathbf{t}=(t_1,t_2)$  -- using Notation~\ref{not:rectangles}. Then the {\bf{rectangular increment}} of $X$ over $R$ 
is defined to be
\begin{align*}
    \square_{\bs\bt}X 
    &= X_{t_{1};t_2}- X_{s_1;t_2} - X_{t_1;s_2} + X_{s_1;s_{2}}.
\end{align*}
\end{defn}

\noindent
Note that, still in line with Notation~\ref{not:rectangles}, we further condense notation by writing 
$$
    X_{\br}:=X_{r_{1};r_{2}},
$$
for $\br=(r_{1},r_{2}) \in \ott^{2}$. 

\noindent 

Recall that the 1D signature is built on the concept of increments. In particular, we observe that $\lla S_{st}(x) ,  (i_{1}) \rra=x^{i_1}_t-x^{i_1}_s$. The rectangular increment defined above naturally extends this to the 2D setting 
\begin{align*}
    \square_{\bs\bt}X = \int_{\Delta^{1}_{[\bs,\bt]}} \partial_{12} X_{s_1;s_2}\ ds_1 ds_2,
\end{align*}
when $X$ is twice continuously differentiable. With this in mind, using Notation~\ref{not:2d-deriv}, we define the following convention for the mixed partial differentials.

\begin{nota}
\label{not:2d-diff}
Consider a function $X \in \cC^{2}(\ott^{2};\RR^{d})$. For $\br=(r_{1},r_{2})$ and $i,j\in\{1,\ldots,d\}$, we define
\begin{equation}
\label{eq:def-2d-diff}
\dd^{i} X_{\br} := \partial_{12}X_{\br}^{i} \, \dd r_{1} \dd r_{2},
\quad\text{and}\quad
\hdd^{ij} X_{\br} := \partial_{1}X_{\br}^{i} \, \partial_{2}X_{\br}^{j} \, \dd r_{1} \dd r_{2} \, ,
\end{equation}
where $X_{\bt}^{i}$ denotes the $i$th component of $X_{\bt} \in \RR^{d}$. Note that in future computations we will write $\dd^{i} X_{\br}$ and $\dd X_{\br}^{i}$ without distinction.

\end{nota}

This subsection is closed by stating a basic change of variables formula in the plane. We employ Einstein's summation convention.

\begin{prop}
Let $X$ be a field in $\cC^{2}(\ott^{2};\RR^{d})$ and consider a $\cC^{2}$-function $f:\RR^{d}\to\RR$. Recall Notation~\ref{not:2d-differentials} used for 2D-simplexes. Then for $(\bs,\bt)\in\Delta_{[\boo,\bT]}^{2}$ we have
\begin{equation}\label{eq:change-variable}
\square_{\bs\bt}f(X)
=
\int_{[\bs,\bt]} \partial_{i}f(X_{\br}) \, \dd^{i}X_{\br} 
+
\int_{[\bs,\bt]} \partial_{ij} f(X_{\br}) \, \hdd^{ij}X_{\br} ,
\end{equation}
where we use convention~\eqref{eq:def-2d-diff} for the differentials $\dd$ and $\hdd$, and where according to notation~\eqref{b4} we set
\begin{equation}
\label{doubleInt}
    \int_{[\bs,\bt]}  
    := 
    \int_{s_1}^{t_1} \,\star\, \int_{s_2}^{t_2} 
    =
    \int_{[s_{1},t_{1}]\times[s_{2},t_{2}]} \, .
\end{equation}
\end{prop}

\begin{rem}\label{rem:fa di bruno}
Formula~\eqref{eq:change-variable} above can also be stated using Fa\`a di Bruno's formula. Namely, by considering $\bs=(s_{1},s_{2})$ as fixed, the field $\bt\mapsto Z_{\bt} = \square_{\bs\bt}f(X)$
is such that
\begin{equation}\label{eq:partial derivatives000}
\partial_{12}Z_{\bt}
=
\partial_{i}f(X_{\bt}) \, \partial_{12}X_{\bt}^{i}
+
\partial_{ij} f(X_{\bt}) \, \partial_{1}X_{\bt}^{i} \, \partial_{2}X_{\bt}^{j} ,
\end{equation}
where we have used Notation \ref{not:2d-deriv} for the right hand side.
\end{rem}


\section{Definition of the 2D signature}
\label{sec:2dsignature}

In this section, we will introduce a notion of 2D signature based on a linear equation, similar to Lemma~\ref{lem:linear-eq-1d}. However, we shall see that this simple definition is not enough to grant basic algebraic properties. Therefore in the second part, we shall propose a new and more fruitful notion of 2D signature. The main algebraic propositions are postponed to later sections.


\subsection {A first notion of 2D signature: the $\id$-signature}
\label{ssec:2dsig-1}

With the notation of Section~\ref{sec:calculus-plane-2} at hand, we now introduce our first notion of 2D signature in terms of a linear equation which is reminiscent of~\eqref{c1}. Namely for a collection of constant matrices $A^{i} \in \RR^{e\times e}, \, i=1,\ldots,d$, an initial condition $v\in\RR^{e}$ and a smooth field $X:\ott^{2}\to\RR^{d}$, we consider the $\RR^{e}$-valued solution $Y$ to the linear equation 
\begin{equation}
\label{d1}
Y_{\bt} = v + \sum_{i=1}^{d} \int_{[\bs,\bt]} A^{i} Y_{\br} \, \dd^{i} \!X_{\br} \, ,
\quad\text{for}\quad \bt\in[\bs,\bT] \, ,
\end{equation}
where we recall from Notation \ref{not:rectangles} that $\bs=(s_1,s_2)$ and $\bT=(T,T)$. Note the double integral~\eqref{doubleInt} as well as the differential \eqref{eq:def-2d-diff} present in \eqref{d1}. The next definition proposes a 2D signature over $X$, seen as a function defined recursively on words similarly to~\eqref{a1}. Recall $\Delta^{2}$ defined 
in Notation~\ref{not:2d-differentials}.

\begin{defn}\label{def:2D sig-1}
Let $X:\ott^{2}\to\RR^{d}$ be a smooth field.
The $\id$-signature $\bS^{\id}(X)$ is a function in $(\cC(\Delta^{2}))^{\cW}$, defined for $(\bs,\bt)\in \Delta^{2}$ and every $w=(i_{1},\ldots,i_{n}) \in \cW$ by
\begin{equation}
\label{d2}
\lla \bS_{\bs\bt}^{\id}(X) , \, w  \rra
:=
\int_{[\bs,\bt]} \lla \bS_{\bs\br}^{\id}(X) , \, (i_{1},\ldots,i_{n-1})  \rra \, \dd^{i_{n}}\! X_{\br}.
\end{equation}
\end{defn}

We illustrate the above definition of the identity signature with the following example. 
\begin{example}\label{ex:multiplicative id sig}
    Let the  field $X:[0,T]^2 \rightarrow \RR^d $ be given as the Hadamard product of two paths  $x^1,x^2\in \operatorname{Lip}([0,T];\RR^d)$, such that  $X$ is given by the vector 
    \[
    X_\bt = \begin{bmatrix}
        x^{1,1}_{t_1} x^{2,1}_{t_2} 
        \\
        \vdots
        \\
        x^{1,d}_{t_1} x^{2,d}_{t_2} 
    \end{bmatrix}.
    \]
 Then for any word $w\in \cW$, it is readily checked that the signature inherits the multiplicative structure of the field $X$, in the sense that 
    \begin{equation*}
        \la \bS^{\id}_{\bs\bt}(X),w\ra = \la S_{s_1t_1}(x^1),w\ra \la S_{s_2t_2}(x^2),w\ra. 
    \end{equation*}
    In \cite{Salvi21} the authors show that the signature kernel  $K(x^1,x^2): = \la S(x^1),S(x^2)\ra_{\mathcal{T}((\RR^d))} $, as used in machine learning \cite{Kirlay19}, solves a Goursat PDE, such that the signature kernel is given as the solution of the PDE
    \begin{equation}
        \frac{\partial^2}{\partial t_1 \partial t_2 } u_\bt = u_\bt  \la \dot{x}^1_{t_1} ,\dot{x}^2_{t_2}\ra ,\quad u_{0,t_2}=u_{t_1,0}=1.  
    \end{equation}
    A formal expansion of the solution to this equation  as the signature kernel can then be written as 
    \begin{equation}
        u_\bt = 1+\sum_{w\in \cW} \la S_{s_1t_1}(x^1),w\ra \la S_{s_2t_2}(x^2),w\ra = 1+\sum_{w\in \cW}  \la \bS^{\id}_{\bs\bt}(X),w\ra. 
    \end{equation}
    Investigations of this type of equation have recently been expanded upon in the setting of Fubini's theorem in \cite{CassPei23}, and  for a general analysis of the  non-linear versions of this  Goursat PDE see \cite{FNH21} in the case of Young fields $X$, and \cite{Bechtold2022NonlinearYE} for a regualrization by noise perspective. 
\end{example}

Analogously to Lemma \ref{lem:linear-eq-1d}, the following lemma shows how the 2D signature over X defined in \eqref{d2} permits to express the $\RR^{e}$-valued solution $Y$ to~\eqref{d1}.  

\begin{lem}\label{lem:ID sig}
Given a smooth field $X:\ott^{2}\to\RR^{d}$, the rectangular increment of the $\RR^{e}$-valued solution $Y$ to~\eqref{d1} can be expanded as
\begin{equation}\label{d3}
\sq_{\bs\bt}Y
=
\sum_{w\in\cW} A^{\circ w} v \lla \bS_{\bs\bt}^{\id}(X) ,  w  \rra.
\end{equation}
Recall that the matrix $A^{\circ w}$ corresponds to the products of matrices in reversed order of the word $w$, introduced in~\eqref{c11}.
\end{lem}

\begin{proof}
According to~\eqref{d1}, the initial conditions for $Y$ on the axes are
\begin{equation*}
Y_{\bs} 
= Y_{s_{1},t_{2}}
=Y_{t_{1},s_{2}}=v \,
\end{equation*}
for all $\bt=(t_{1},t_{2})\in[\bs,\bT]$. Therefore, for all $\bt\in[\bs,\bT]$ we also have
\begin{equation*}
    \sq_{\bs\bt}Y = Y_{\bt} - v .
\end{equation*}
Therefore, \eqref{d1} becomes
\begin{equation}
\label{d4}
\sq_{\bs\bt}Y
=
\sum_{i=1}^{d} \int_{[{\bs},{\bt}]} A^{i} Y_{\br} \, \dd^{i}\! X_{\br} .
\end{equation}
We now proceed as in the proof of Lemma~\ref{lem:linear-eq-1d}. Namely on the right hand side of~\eqref{d4} we write $Y_{\br}=v+\sq_{\bs,\br}Y$. Using Einstein's summation convention we get 
\begin{equation}
\label{d41}
\sq_{\bs\bt}Y
=
A^{i_{1}} v \int_{[{\bs},{\bt}]} \dd^{i_{1}}\!  X_{\br^1}+ R_{\bs\bt}^{1} \, ,
\quad\text{where}\quad
R_{\bs\bt}^{1}
=
\int_{[{\bs},{\bt}]} A^{i_{1}} \!\lp\sq_{\bs\br^1}Y\rp \, \dd^{i_{1}}\!  X_{\br^1}  .
\end{equation}
Iterating \eqref{d4} and \eqref{d41} as in the proof of Lemma~\ref{lem:linear-eq-1d}
confirms a ``formal expansion'' \eqref{d3}.
Once we have 
shown \Cref{prop:decay},
this formal expansion
becomes an honest expansion.
\end{proof}

\begin{rem}
The signature $\bS_{\bs\bt}^{\id}(X)$ can be written more explicitly
as a tensor series in $\cT((\mathbb{R}^{d}))$ with
multiple integrals as coefficients.
Indeed, using Notation~\ref{eq:def-2d-simplex} for simplexes, recursion~\eqref{d2} yields
\begin{equation}
\label{d5}
\lla \bS_{\bs\bt}^{\id}(X) , \, w  \rra
=
\int_{\Delta_{[\bs,\bt]}^{n}} \dd^{i_{1}} X_{\br^{1}} \cdots \dd^{i_{n}} X_{\br^{n}} \, , 
\end{equation}
where we recall \eqref{eq:def-2d-simplex} for $\Delta^n_{[\bs,\bt]}$.
\end{rem}

The signature $\bS_{\bs\bt}^{\id}(X)$ solves a linear differential equation which is reminiscent of \eqref{eq:path Sig diff}.
We state this infinite-dimensional
variant of \eqref{d1} below. 
\begin{prop}\label{prop:Id sig eq}
Let $X$ be a function in $\cC^2([0,T]^2;\RR^d)$. Recall the definition of tensor series $\cT((\RR^d))$ spelled out in Section \ref{sec:prelim}. Then as a $\cT((\RR^d))$-valued function indexed by $[0,T]^2$, the signature $\bS^{\id}(X)$ satisfies the following equation for all $(\bs,\bt)\in \Delta^2_{[\mathbf{0},\bT]}$: 
    \begin{equation}
        \bS_{\bs\bt}^{\id}(X)= \mathbf{1}+ \int_{[\bs,\bt]} \bS_{\bs\br}^{\id}(X) \otimes  \dd X_{\br}.
    \end{equation}
\end{prop}

\begin{proof}
We denote by $\pi_n(\bS^{\id}_{\bs\bt}(X))$ the projection of $\bS^{\id}_{\bs\bt}(X)$ onto the set of level $n$ signature components. More specifically, using Definition \ref{def:words} and our notation in \eqref{c4} we set 
\begin{equation}
    \pi_n(\bS^{\id}_{\bs\bt}(X))= \sum_{w\in \cW_n} \la \bS^{\id}_{\bs\bt}(X),w\ra e_w. 
\end{equation}
    Then resorting to our explicit representation \eqref{d5} we have 
\begin{equation}
\label{eq:projection}
    \pi_n(\bS^{\id}_{\bs\bt}(X)) = \int_{\Delta^n_{[\bs,\bt]}} \bigotimes_{i=1}^n \dd X_{\br^i}, 
\end{equation}
where $\dd X = (\dd X^1,\ldots, \dd X^d)$. Moreover, one can recast \eqref{eq:projection} as 
\[
 \pi_n(\bS^{\id}_{\bs\bt}(X)) = \int_{[\bs,\bt]}\left(\int_{\Delta^{n-1}_{[\bs,\br^n]}} \bigotimes_{i=1}^{n-1} \dd X_{\br^i}\right)\otimes \dd X_{\br^n} .
\]
Resorting to the explicit representation in \eqref{d5} again, we end up with 
\[
   \pi_n(\bS^{\id}_{\bs\bt}(X)) = \int_{[\bs,\bt]}\pi_{n-1}(\bS^{\id}_{\bs\br^n}(X))\otimes \dd X_{\br^n}.
\]
Putting together all the projections finishes the proof. 
\end{proof}

\subsection{Chen's relation: Horizontal and vertical}
\label{sec:2DChen}

Recall that Chen's relation for curves indexed by $[0,T]$ has been stated in Lemma \ref{lem:Chens integral relation}. We now investigate proper generalizations to the 2D-setting. Similarly to other contributions in the field (see \cite{giusti2022topological,HT}) we will only get partial versions of this relation for the 2D signature. Our restrictions will be of two types: 
\begin{enumerate}[wide, labelwidth=!, labelindent=0pt, label= \textbf{(\roman*)}]
\setlength\itemsep{.05in}
    \item We will only treat partial Chen-type relations in directions $1$ and $2$ separately. 
    One can combine the two, but will then obtain a rather complicated relation. 
    
    \item Our results will be restricted to the signature $\bS^{\id}(X)$ introduced in Section \ref{ssec:2dsig-1}. 
\end{enumerate}
 In subsequent sections, we shall give some hints about the way to overcome those difficulties, see in particular Remark \ref{rem: chen to full}. 
 
 Below we start by introducing some useful notation to state our main results. 

\begin{nota}
\label{not:shuffle-on-words}
Consider two words $w=(i_{1},\ldots,i_{n})$ and $w'=(j_{1},\ldots,j_{n'})$ in $\cW$, where $\cW$ was defined in Definition \ref{def:words}. The concatenation of $w$ and $w'$ is denoted by $\hw=[ww']$. Then, if $\rho\in\ssi_{\{1,\ldots,n\}}$ we set
\begin{equation*}
w_{\rho}
=
\lp i_{\rho(1)},\ldots,i_{\rho(n)}  \rp.
\end{equation*}
\end{nota}
We continue with a definition of two convenient types of products between integral operators. 
\begin{defn}
Let $X$ be a field in $\cC^2([0,T]^2)$ and consider the signature $\bS^{\id}(X)$ given in \eqref{d2}. Pick two words $w\in \cW_n$ and $w'\in \cW_{n'}$ with $n,n'\in \NN$. For $a<b$, $c<d$ and $(\bs,\bt)\in \Delta^2_{[0,T]^2}$ we introduce the following two convolution products: 
\begin{enumerate}[wide, labelwidth=!, labelindent=0pt, label= \textbf{(\roman*)}]
\setlength\itemsep{.05in}
    \item a horizontal convolution product of the form \begin{multline}\label{eq:horizontal convolution}
    \la \bS^{\id}_{(a,s_2),(b,t_2)}(X),w\ra \ast_1 \la \bS^{\id}_{(c,s_2),(d,\cdot)}(X),w'\ra \\
    :=\int_{\Delta^{l}_{a,b} \times \Delta^{l}_{s_2,t_2}} \left(\int_{\Delta^{k}_{c,d} \times \Delta^{k}_{s_2,r_2^{l+1}}} 
 \prod_{j=1}^l \dd X^{w'_j}_{\br^j}\right) \prod_{j=l+1}^{n+k} \dd X^{w_{j-l}}_{\br^j} .
\end{multline}
   \item 
   a vertical convolution product of the form 
   \begin{multline}\label{eq:vertical convolution}
        \la \bS^{\id}_{(s_1,a),(t_1,b)}(X),w\ra \ast_2 \la \bS^{\id}_{(s_1,c),(\cdot,d)}(X),w'\ra \\
        :=\int_{\Delta^{l}_{s_1,t_1} \times \Delta^{l}_{a,b}} \left(\int_{\Delta^{k}_{s_1,r_1^{l+1}} \times \Delta^{k}_{c,d}} 
 \prod_{j=1}^l \dd X^{w'_j}_{\br^j}\right) \prod_{j=l+1}^{n+k} \dd X^{w_{j-l}}_{\br^j} .
   \end{multline}
    \end{enumerate}
\end{defn}

\begin{rem}
    In terms of integrals over simplexes the horizontal convolution product is obtained by integrating the $1$-variable over $\Delta^n_{a,b}\times \Delta^{n'}_{c,d}$, while the $2$-variable is integrated over the whole $n+n'$ dimensional simplex $\Delta^{n+n'}_{s_2,t_2}$. In terms of our integral operators from Definition \ref{def:intg-operator} one can write 
    \begin{equation}\label{g1}
    \la \bS^{\id}_{(a,s_2),(b,t_2)}(X),w\ra \ast_1 \la \bS^{\id}_{(c,s_2),(d,\cdot)}(X),w'\ra
    =\left(\int_{\Delta^{n}_{a,b}} \star \int_{\Delta^{n'}_{c,d}}  \right)\starI \int_{\Delta^{n+n'}_{s_2,t_2}} \prod_{j=1}^{n+n'}\dd X^{[w,w']_j}_{\br^j} .
\end{equation}
    Notice that in relation \eqref{g1}, the glueing of the variables in direction $2$ is reminiscent of the convolution structure for Volterra equations introduced in \cite{HT}. Moreover, one can also express the horizontal convolution product~\eqref{eq:horizontal convolution} using signature elements as follows:
    \begin{eqnarray*}
             \la \bS^{\id}_{(a,s_2),(b,t_2)}(X),w\ra \ast_1 \la \bS^{\id}_{(c,s_2),(d,\cdot)}(X),w'\ra
             &= &
             \int_{\Delta^{l}_{a,b} \times \Delta^{l}_{s_2,t_2}} \la \bS^{\id}_{(c,s_2),(d,r_2^{n'+1})}(X),w'\ra \prod_{j=n'+1}^{n+n'} \dd X^{w_{j-n'}}_{\br^j}
             \notag\\
             &=&
             \int_{\Delta^{l}_{a,b} \times \Delta^{l}_{s_2,t_2}} \la \bS^{\id}_{(c,s_2),(d,r_2^{1})}(X),w'\ra \prod_{j=1}^{n} \dd X^{w_{j}}_{\br^j}. 
    \end{eqnarray*}
    The vertical product $\ast_2$ can be considered similarly. 
\end{rem}

We will express our partial Chen type relations within the framework of algebraic integration as in Remark \ref{rem:algebraic integration}. Since we are now in a setting with two variables, we introduce some notation about directional versions of the operator $\delta$. 

\begin{nota}\label{not:delta}
For functions or increments defined on $[0,T]^{2}$
we will denote by $\delta^1$ the delta operator $\delta$ introduced in Remark~\ref{rem:algebraic integration},  restricted to act on the $1$-variable only. Similarly, we denote by $\delta^2$ the $\delta$ operator restricted to the $2$-variable. 
\end{nota}

\begin{rem}
    It is readily checked that we can compose the two delta functions as $\delta^1\circ \delta^2$, and that the composition is commutative. Furthermore, it is clear that, with the definition from~\eqref{eq:delta def} in mind, for every field $X$ in $\cC^2([0,T]^2)$ we have 
    \[
    \delta_{t_1}^1\delta_{t_2}^2 X_\bs = \square_{\bs\bt}X. 
    \]
    
\end{rem}

We are now ready to state our partial Chen-type result in the horizontal and vertical directions separately. 
\begin{prop}[Horizontal and vertical Chen's relation]\label{Prop:Horizontal and vertical chen}
    Consider a field $X\in \cC^2([0,T]^2)$. let $\bS^{\id}(X)$ be its associated 2D-signature given in Lemma \ref{lem:ID sig}. Pick a word $w\in \cW_n$ for $n\geq 1$. Then for $(\bs,\bt)\in \Delta^2_{[0,T]^2}$ the following two relations holds: 
   \begin{enumerate}[wide, labelwidth=!, labelindent=0pt, label= \textbf{(\roman*)}]
\setlength\itemsep{.05in}
       \item  For any $u_1\in [s_1,t_1]$ we have 
    \begin{equation*}
        \la \delta^1_{u_1}\bS^{\id}_{\bs\bt}(X),w\ra  
        = \sum_{k=1}^{n-1}  \la \bS^{\id}_{(u_1,s_2),\bt}(X),(w_1\cdots w_k)\ra \ast_1  \la \bS^{\id}_{\bs,(u_1,\cdot)}(X),(w_{k+1}\cdots w_n)\ra , 
    \end{equation*}
    where the convolution product $\ast_1$ comes from \eqref{eq:horizontal convolution}. 
    \item  For any $u_2\in [s_2,t_2]$ we have 
    \begin{equation*}
        \la \delta^2_{u_2}\bS^{\id}_{\bs\bt}(X),w\ra  
        = \sum_{k=1}^{n-1}  \la \bS^{\id}_{(s_1,u_2),\bt}(X),(w_1\cdots w_k)\ra \ast_2  \la \bS^{\id}_{\bs,(\cdot,u_2)}(X),(w_{k+1}\cdots w_n)\ra , 
    \end{equation*}
    where the convolution product $\ast_2$ comes from \eqref{eq:vertical convolution}. 
   \end{enumerate} 
   
\end{prop}

\begin{proof}
For simplicity, we perform the proof for the horizontal convolution $\ast_1$, but the vertical convolution $\ast_2$ follows by the same type of arguments.   We use the representation \eqref{eq:rep of 2d sig with simplex permutation} of the signature with $\nu=\id$, which yields 
    \begin{equation}\label{eq:proof 1}
        \la \bS^{\id}_{\bs\bt}(X),w\ra = \int_{\Delta^n_{s_1,t_1} } \starI \int_{\Delta^n_{s_2,t_2}}  \prod_{i=1}^n \dd X^{w_i}_{\br^i}. 
    \end{equation}
    Hence, applying the operator $\delta^1$ from notation \ref{not:delta} to both sides of \eqref{eq:proof 1} we get 
    \[
    \la \delta^1_{u_1} \bS^{\id}_{\bs\bt}(X),w\ra = \lp \delta^1_{u_1}\int_{\Delta^n_{s_1,t_1} } \rp \starI \int_{\Delta^n_{s_2,t_2}}  \prod_{i=1}^n \dd X^{w_i}_{\br^i}.
    \]
    We can now resort to \eqref{eq:ff}, which yields 
    \begin{equation}\label{eq:proof 2}
        \la \delta^1_{u_1} \bS^{\id}_{\bs\bt}(X),w\ra = \sum_{l+k=n \atop  l,k\geq 1} \lp \int_{\Delta^l_{s_1,u_1} } \star \int_{\Delta^k_{u_1,t_1} }\rp  \starI \int_{\Delta^n_{s_2,t_2}}  \prod_{i=1}^n \dd X^{w_i}_{\br^i}.
    \end{equation}
    Owing to our representation of the horizontal convolution product in \eqref{eq:horizontal convolution} we obtain the claimed relation.  
\end{proof}
\subsection{Challenges with the $\id$ signature}

The signature proposed in~\eqref{d2} is rather natural and interestingly simple. However, it lacks some of the basic properties compared to the signature of a path. Let us briefly highlight two of those shortcomings:
\begin{enumerate}[wide, labelwidth=!, labelindent=0pt, label= \textbf{(\roman*)}]
\setlength\itemsep{.05in}

\item
\label{it:shuffle} \emph{Lack of shuffle property.}
A common requirement for a proper signature is stability under multiplication, i.e., iterated integrals \eqref{d5} should form an algebra. However,  $\bS^{\id}(X)$ does not fulfill this property; the $d=1$ case, i.e., considering a real-valued field $X \in \mathcal{C}^ {2}(\ott^{2};\RR)$ suffices to see this. Indeed, due to~\eqref{d5}, the increment $\square X=\int \dd X$ is trivially seen as a component of $\bS^{\id}(X)$. Now define the product
\begin{equation*}
\Pi_{\bs\bt}
=\lp \int_{[\bs,\bt]} \dd X_{\br}\rp^2 = 
\int_{[\bs,\bt]} \dd X_{\br} \, \int_{[\bs,\bt]} \dd X_{\bv} , 
\end{equation*}
where we recall that $\dd X$ is given by~\eqref{eq:def-2d-diff}. Then some elementary computations involving integration in $[\bs,\bt]\times[\bs,\bt]$ reveal that $\Pi_{\bs\bt}={{2}}\Pi_{\bs\bt}^{1} + {{2}} \Pi_{\bs\bt}^{2}$, with
\begin{equation}\label{d6}
\Pi_{\bs\bt}^{1}
=
\int_{\Delta_{[\bs,\bt]}^{2}} \dd X_{r_{1}^{1};r_{2}^{1}} \dd X_{r_{1}^{2};r_{2}^{2}} \,
\quad\text{and}\quad
\Pi_{\bs\bt}^{2}
=
\int_{\Delta_{[\bs,\bt]}^{2}} \dd X_{r_{1}^{1};r_{2}^{2}} \dd X_{r_{1}^{2};r_{2}^{1}} \, .
\end{equation}
Now referring to~\eqref{d5} again, the term $\Pi_{\bs\bt}^{1}$ in~\eqref{d6} is easily identified with $\langle \bS^{\id}(X), (1,1)\rangle_{\bs\bt}$. However, $\Pi_{\bs\bt}^{2}$ is not part of the signature, due to the permutation of $r_{2}^{1}$ and $r_{2}^{2}$ in the double integrals. This permutation phenomenon will feature prominently in our future considerations.

\item\label{it:ch of vb} 
\emph{Lack of change of variables formula for exponential functions.}
A signature should accommodate simple expansions in change of variables formulae. Here again, it is easy to find examples for which $\bS^{\id}(X)$ fails to be appropriate for this elementary task. That is, consider the same generic real-valued field $X \in \mathcal{C}^ {2}(\ott^{2};\RR)$ as in item~\ref{it:shuffle} above. Next, for $\bt\in [\bs,\bT]$ set $Z_{\bt}=\exp(\square_{\bs\bt}X)$. 
Hence, applying a small variant of the  Fa\'a Di Bruno type  formula  stated in Remark \ref{rem:fa di bruno} (differentiating  $\bt\mapsto f(\square_{\bs\bt} X)=Z_{\bt}$ instead of $\bt \mapsto f(X_\bt)$), we have that 
\[
\partial_{12}Z_{\bt} = Z_{\bt} \, \partial_{12} X 
+ Z_{\bt}\, \partial_1\left(X_{\bt}-X_{t_1,s_2}\right)\partial_2\left( X_{\bt}-X_{s_1,t_2}\right).
\]
Therefore, a first-order approximation for the rectangular increment of $Z$ is
\begin{equation}\label{d7}
\square_{\bs\bt} Z
\approx
Z_{\bs} \,
\square_{\bs\bt} X
+
Z_{\bs} \,R^{X}_{\bs\bt},
\end{equation}
where the term $R^{X}_{\bs\bt}$ is defined by 
\[
R^{X}_{\bs\bt} = \int_{[\bs,\bt]} \partial_1\left(X_{\br}-X_{r_1,s_2}\right)\partial_2\left( X_{\br}-X_{s_1,r_2}\right)\, \dd\br.  
\]
Now, as mentioned in~\ref{it:shuffle}, the increment $\square_{\bs\bt} X$ is part of $\bS^{\id}(X)$, {{i.e., $\square_{\bs\bt} X=\langle \bS^{\id}_{\bs\bt}(X), (1)\rangle$}}. However, the last term $R^{X}$ in~\eqref{d7} is not an element of the signature $\bS^{\id}(X)$. The above first-order approximation has thus to be expressed in terms of a larger object then the $\id$ signature, which we will call the full signature.
\end{enumerate}


\subsection{Definition of the full 2D signature} 

In Section \ref{ssec:2dsig-1}, we have seen that the $\id$-signature $\bS^{\id}(X)$
does not possess some of the desirable properties known to hold for the classical path signatures.
As addressed in relation \eqref{d6}, the problem seems to arise from the fact that permutations of simplexes appear. To address this issue, we will now propose another notion of 2D signature, taking such permutations into account.  

Before proceeding to the definition of the 2D signature, we will introduce a new set of word-permutation tuples. Recall that the set of words $\cW_n$ of length $n\geq 0$ is defined in Definition \ref{def:words}. 

\begin{defn}\label{def:2d words}
The set of extended words is defined as $\hat{\cW} = \bigcup_{n=0}^\infty \hat{\cW}_n $, where $\hat{\cW}_0=\{(\e,\id)\}$ and for $n\geq 1$
\begin{equation}
    \hat{\cW}_n=\{(w,\nu)\ | \,w\in \cW_n,\, \nu\in \Sigma_{\{1,\ldots,n\}}\}. 
\end{equation}
\end{defn}

Recall that $\Delta^2$ is given in Notation \ref{not:2d-differentials}.  
We are now ready to define the full 2D signature. 

\begin{defn}\label{def:extended Sig}
Let $X$ be a function in $\cC^2([0,T]^2;\RR^d)$.
The {\em (full) 2D signature} $\bS(X)$ is a function in $(\cC(\Delta^2))^{\hat{\cW}}$, defined for an extended word $(w,\nu)\in \hat{\cW}_{n}$ by setting 
\begin{equation}
\label{eq:2d pairing}
    \langle\bS_{\bs\bt}(X),(w,\nu)\rangle  := \int_{\Delta^n_{[\bs,\bt]}} \prod_{i=1}^n \dd X^{w_i}_{r_1^i,r_2^{\nu_i}}. 
\end{equation}
\end{defn}

\begin{rem}
    According to \eqref{b2} and \eqref{interlace sets}, one can recast \eqref{eq:2d pairing} as 
    \begin{equation}
    \label{eq:rep of 2d sig with simplex permutation}
        \langle\bS_{\bs\bt}(X),(w,\nu)\rangle = \int_{\Delta^n_{s_1,t_1}}\starI \int_{\Delta^n_{s_2,t_2}\nu^{-1}} \prod_{i=1}^n \dd X^{w_i}_{r_1^i,r_2^i}. 
    \end{equation}
    The representation in \eqref{eq:rep of 2d sig with simplex permutation} will be useful to prove some of our algebraic relations. 
    %
\end{rem}

With the definition of the full 2D signature at hand, let us illustrate its behavior concerning the 2D signature $\bS^{\id}(X)$ proposed in Definition \ref{def:2D sig-1} and its relation to 2D iterated sums as proposed in \cite{diehl2022two}. 

\begin{example}
    Let us illustrate the 2D signature in light of the multiplicative field discussed in Example \ref{ex:multiplicative id sig}. Namely we consider a $\RR$-valued field $X_{\bt}=x^1_{t_1}x^2_{t_2}$ for two Lipschitz paths $x^1$ and $x^2$. Choosing an extended word $(w,\nu)\in \hat{\cW}_n$ we see that 
    \begin{equation}
        \la \bS_{\bs\bt}(X),(w,\nu)\ra =\la S_{s_1t_1}(x^1),w\ra \la S_{s_2t_2}(x^2),w_{\nu^{-1}}\ra, 
    \end{equation}
    where $w_{\nu^{-1}}=(w_{\nu^{-1}(1)},\ldots,w_{\nu^{-1}(n)})$. Indeed, to see this observe that 
    \begin{equation}
        \la \bS_{\bs\bt}(X),(w,\nu)\ra = \la S_{s_1t_1}(x^1),w\ra \int_{\Delta^n_{s_2t_2}} \prod_{i=1}^n \dd x^{2,w_i}_{r^{\nu_i}_2},
    \end{equation}
    and by Fubini's theorem it follows that 
    \[
    \int_{\Delta^n_{s_2t_2}} \prod_{i=1}^n \dd x^{2,w_i}_{r^{\nu_i}_2}= \int_{\Delta^n_{s_2t_2}} \prod_{i=1}^n \dd x^{2,w_{\nu^{-1}_i}}_{r^{i}_2} = \la S_{s_2t_2}(x^2),w_{\nu^{-1}}\ra. 
    \]
    We see that when $\nu=\id$, then we recover the results of Example \ref{ex:multiplicative id sig}. 
\end{example}

\begin{example}
Consider $i_1,i_2,i_3\in \{1,\ldots,d\}$ and the permutation $\perm{132} \in \Sigma_{\{1,2,3\}}$.
\begin{align*}
    \langle\bS_{\bs\bt}(X),((i_1,i_2,i_3),\perm{132})\rangle  =
    \int_{
    \substack{
        s_1 < r_1^1<r_1^2<r_1^3 < t_1\\
        s_2 < r_2^1<r_2^2<r_2^3 < t_2}}
        \dd X^{i_1}_{r^1_1,r^1_2}
        \dd X^{i_2}_{r^2_1,r^3_2}
        \dd X^{i_3}_{r^3_1,r^2_2}.
\end{align*}
Employing a matrix notation akin to the one used in \cite{diehl2022two}\footnote{Contrary to what has been done in \cite{diehl2022two}, we orient the plane, and correspondingly the matrices, from bottom-left to top-right.} this corresponds to the matrix
\begin{align*}
       \left\langle\bS_{\bs\bt}(X), \begin{bmatrix}
           0    &i_2    & 0 \\
           0    &0      & i_3 \\
           i_1  &0      & 0
        \end{bmatrix}\right\rangle.
\end{align*}    
\end{example}

\begin{rem}\label{rem:two issues of ID sig}
The two issues \ref{it:shuffle}-\ref{it:ch of vb} raised after Proposition  \ref{prop:Id sig eq} are easily resolved by using the full 2D signature $\bS(X)$. Specifically, the term $\Pi^2_{\bs\bt}$ in \eqref{d6} can be written as 
\begin{equation}\label{eq:d9}
    \Pi^2_{\bs\bt} = \int_{\Delta^2_{[\bs,\bt]}} \dd X_{r_1^1,r_2^2}\dd X_{r_1^2,r_2^1}= \la \bS_{\bs\bt}(X),(w,\nu)\ra,
\end{equation}
where $w=(1,1)$ and $\nu\in \Sigma_{\{1,2\}}$ is defined by $\nu=\perm{21}$. As for the term $R^{X}_{\bs\bt}$ in \eqref{d7}, observe that by applying the fundamental theorem of calculus to the increments in $$
\partial_1\left(X_{\br}-X_{r_1,s_2}\right)\partial_2\left( X_{\br}-X_{s_1,r_2}\right),
$$
we also get that 
\[
    R^{X}_{\bs\bt} =\int_{\Delta^2_{[\bs,\bt]}} \dd X_{r_1^1,r_2^2}\dd X_{r_1^2,r_2^1}= \la \bS_{\bs\bt}(X),((1,1),\perm{21})\ra= \Pi_{\bs\bt}^2.
\] 
With this computation we also observe that even for fields with one ambient dimension, the full signature is not trivial (i.e.~it can not be written in terms of monomials), in contrast to what is the case for the path signature. 
\end{rem}

\begin{rem}\label{rem: chen to full}
    It is a non-trivial task to generalize Proposition \ref{Prop:Horizontal and vertical chen} to the 2D signature $\bS(X)$ from Definition \ref{def:extended Sig}. Indeed, consider $w\in \cW_n$, $\nu \in \Sigma_{\{1,\ldots,n\}}$ and $s_1,t_1,u_1$ as in Proposition \ref{Prop:Horizontal and vertical chen}. One can repeat the manipulations from \eqref{eq:proof 1} and \eqref{eq:proof 2} in order to get 
    \[
    \la \bS^{\id}_{\bs\bt}(X),w\ra = \sum_{l+k=n \atop  l,k\geq 1} \lp \int_{\Delta^l_{s_1,u_1} } \star \int_{\Delta^k_{u_1,t_1} }\rp  \starI \int_{\Delta^n_{s_2,t_2}\nu^{-1}}  \prod_{i=1}^n \dd X^{w_i}_{\br^i}.
    \]
   The next step would then be to decompose $\Delta^n_{s_2,t_2}\nu^{-1}$ further.
   Now the problem is that for $i=1,\ldots,k$, the integration variables $r_2^i$ will in general not be in the simplex $\Delta^k_{s_2,u_2}$. This rules out the possibility of deriving a formula like \eqref{eq:horizontal convolution}. The ability to overcome this difficulty relies on a richer algebraic structure around decompositions of simplexes, which is outside the scope of the current article (but which we hope will be investigated in the future). 
\end{rem}

We conclude this section by presenting a linear integral equation that is satisfied by the truncation of the 2D signature $\bS(X)$ to words of length $n \le 3$.

\begin{prop}
\label{prop:recursive rel for some permutations}
Let $X:[0,T]^2 \rightarrow \RR^d $ be a $\cC^2$ field. Consider the 2D signature $\bS(X)$ given in Definition~\ref{def:extended Sig}, seen as an element of $\cC(\Delta^2)^{\hat{\cW}}$. Then for $n\leq 3$ and  $(w,\nu)\in \hat{\cW}_n$, with $w=(i_1,\ldots,i_n)$, the coordinate $\la \bS(X),(w,\nu)\ra $ satisfies an equation of the form
\begin{equation}\label{e1}
    \la \bS_{\bs\bt}(X),(w,\nu)\ra = \mathbf{1}_{\{w=\e\}} + \mathbf{1}_{\{w\ne\e\}}  \int_{[\bs,\bt]} \cL_{\bv}^{\bs\bt;(i_1\dots i_{n-1},\nu)} \dd X^{i_n}_{\bv},
\end{equation}
where the function $\cL_{\bv}^{\bs\bt;(i_1\dots i_{n-1},\nu)}$ is given by the following expressions for various permutations $\nu$: 
{\small
\begin{align*}
   \cL_{\bv}^{\bs\bt;(\e,\perm{1})} 
   &= 1 \\
   \cL_{\bv}^{\bs\bt;(i_1,\perm{12})} 
   &=
   \la \bS_{s_{1},s_{2};v_1,v_2}(X), ((i_1), \id) \ra \\
   \cL_{\bv}^{\bs\bt;(i_1,\perm{21})} 
   &=
   \la \bS_{s_{1},v_{2};v_1,t_2}(X), ((i_1), \id) \ra \\
   \cL_{\bv}^{\bs\bt;(i_1i_2,\perm{123})} &= \la \bS_{s_{1},s_{2};v_1,v_2}(X), ((i_1,i_2), \id) \ra \\
   \cL_{\bv}^{\bs\bt;(i_1i_2,\perm{132})} 
   &=
   \la \bS_{s_{1},s_{2};v_1,t_2}(X), ((i_1,i_2), \id) \ra - \la \bS_{s_{1},s_{2};v_{1},v_{2}}(X), ((i_1,i_2), \id) \ra - \la \bS_{s_1,v_2;v_1,t_2}(X), ((i_1,i_2), \id) \ra \\
   \cL_{\bv}^{\bs\bt;(i_1i_2,\perm{213})}
   &= 
   \la \bS_{s_{1},s_{2};v_1,v_2}(X), ((i_1,i_2), \perm{21}) \ra \\
   \cL_{\bv}^{\bs\bt;(i_1i_2,\perm{231})}
   &= 
   \la \bS_{s_{1},v_{2};v_1,t_2}(X), ((i_1,i_2), \perm{12}) \ra \\
   \cL_{\bv}^{\bs\bt;(i_1i_2,\perm{321})} &= \la \bS_{s_{1},v_{2};v_1,t_2}(X), ((i_1,i_2), \perm{21}) \ra, 
\end{align*}
}
and 
\small{
\begin{multline*}
       \cL_{\bv}^{\bs\bt;(i_1i_2,\perm{312})}
   = 
   \la \bS_{s_{1},s_{2};v_1,t_2}(X), ((i_1,i_2), \perm{21}) \ra - \\ 
   \la \bS_{s_{1},s_{2};v_{1},v_{2}}(X), ((i_1,i_2), \perm{21}) \ra - \la \bS_{s_1,v_2;v_1,t_2}(X), ((i_1,i_2), \perm{21}) \ra .
\end{multline*}
}

\end{prop}

\begin{proof}
    The cases $n=1,2$ are easily handled, similarly to Proposition \ref{prop:Id sig eq}. The same kind of argument also holds for $n=3$ when $\nu(3)\in \{1,3\}$.
    Therefore, in the remainder of the proof, we will focus on the case $n=3$ and $\nu(3)=2$. Specifically, let us assume $w=(i_1,i_2,i_3)$ for $i_1,i_2,i_3 \in \{1,\ldots,d\}$ and
    $\nu = \perm{132}$.
    The other case, $\nu=\perm{312}$, can be treated similarly. The value of the 2D signature is defined for 
    $(\bs,\bt)\in \Delta^2_{[\bs,\bt]}$:
    \[
        \la \bS_{\bs\bt}(X), (w,\nu)\ra 
        = \int_{\Delta^3_{[\bs,\bt]}} \dd X^{i_1}_{r_1^1,r_2^1} \dd X^{i_2}_{r_1^2,r^3_2} \dd X^{i_3}_{r_1^3,r_2^2}. 
    \]
    Recall that $\Delta^3_{[\bs,\bt]}$ was introduced in \eqref{eq:def-2d-simplex}. To set up an induction, let us write  
    \[
        \la \bS_{\bs\bt}(X), (w,\nu)\ra 
        = \int_{{s_1 \leq r_1^3 \leq t_1}\atop {s_2 \leq r_2^2 \leq t_2}} \left( \int_{{s_1\leq r_1^1\leq r_1^2 \leq r_1^3}\atop{s_2 \leq r_2^1 \leq r_2^2\, ; \, r_2^2\leq r_2^3 \leq t_2}} \dd X^{i_1}_{r_1^1,r_2^1} \dd X^{i_2}_{r_1^2,r^3_2} \right) \dd X^{i_3}_{r_1^3,r_2^2}. 
    \]
    Otherwise stated, one has 
    \begin{equation}
    \label{eq:alt Q}
        \la \bS_{\bs\bt}(X), (w,\nu)\ra 
        = 
        \int_{{s_1 \leq r_1^3 \leq t_1}\atop {s_2 \leq r_2^2 \leq t_2}} Q_{r_1^3,r_2^2}  \dd X^{i_3}_{r_1^3,r_2^2},  
    \end{equation}
    where for $\bv=(v_{1},v_{2})\in[\bs,\bt]$ we have set 
    \begin{equation}
    \label{eq:Q def}
    Q_{v_1,v_2} 
    := \int_{(r_1^1,r_1^2,r_2^1,r_2^3 )\in D_{v_1,v_2}} \dd X^{i_1}_{r_1^1,r_2^1}\dd X^{i_2}_{r_1^2,r_2^3},  
    \end{equation}
    and where the domain $D_{v_{1},v_{2}}\subset [0,T]^4$ can be decomposed as $D_{v_{1},v_{2}}=D^1_{v_{1}} \times D^2_{v_{2}}$ with 
    \begin{align*}
        D^1_{v_{1}} &= \{(p_1,q_1)\in \Delta^2_{s_1,t_1}; s_1\leq p_1 \leq q_1 \leq v_{1}\}
        \\
            D^2_{v_{2}} &= \{(p_2,q_2)\in \Delta^2_{s_2,t_2}; s_2\leq p_2 \leq v_{2}, \text{ and } v_{2} \leq q_2 \leq  t_2\}. 
    \end{align*}
    Notice that we have several dummy integration variables above, so we will simplify relations \eqref{eq:alt Q}-\eqref{eq:Q def} as 
    \begin{equation}\label{eq:f}
        \la \bS_{\bs\bt}(X), (w,\nu)\ra = \int_{{s_1 \leq v_1 \leq t_1}\atop {s_2 \leq v_2 \leq t_2}} Q_{v_1,v_2}  \dd X^{i_3}_{v_1,v_2} =\int_{[\bs,\bt]} Q_{\bv}  \dd X^{i_3}_{\bv},  
    \end{equation}
    with 
    \begin{equation*}
        Q_\bv = \int_{(\br^1,\br^2)\in D_{\bv}} \dd X^{i_1}_{\br^1}\dd X^{i_2}_{\br^2}. 
    \end{equation*}
    Now notice that due to relation \eqref{eq:special chen 2 order} we have 
    \begin{equation*}
        \int_{D^2_{v_2}} 
        = \int_{\Delta^1_{s_2,v_2}}\star \int_{\Delta^1_{v_2,t_2}} 
        = \int_{\Delta^2_{s_2,t_2}}
                -\int_{\Delta^2_{s_2,v_2}}-\int_{\Delta^2_{v_2,t_2}}. 
    \end{equation*}
    This information is then inserted into the definition of $Q$ to discover that 
    \begin{align}\label{eq:Q rel w S}
     Q_\bv 
    &=  \int_{\Delta^2_{[s_1,v_1]\times [s_2,t_2]}} \dd X^{i_1}_{\br^1}\dd X^{i_2}_{\br^2} 
                - \int_{\Delta^2_{[s_1,v_1]\times [s_2,v_2]}} \dd X^{i_1}_{\br^1}\dd X^{i_2}_{\br^2} 
                - \int_{\Delta^2_{[s_1,v_1]\times [v_2,t_2]}} \dd X^{i_1}_{\br^1}\dd X^{i_2}_{\br^2}  
        \\ 
        &=  \la \bS_{s_{1},s_{2};v_1,t_2}(X), ((i_1,i_2), \id) \ra 
        - \la \bS_{s_{1},s_{2};v_1,v_2}(X), ((i_1,i_2), \id) \ra
    - \la \bS_{s_{1},s_{2};v_1,t_2}(X), ((i_1,i_2), \id) \ra.       \nonumber   
    \end{align}
    Plugging \eqref{eq:Q rel w S} into \eqref{eq:f}, we thus end up with 
    \begin{multline*}
        \la \bS_{\bs\bt}(X), (w,\nu)\ra 
        = \int_{[\bs,\bt]} \la \bS_{s_{1},s_{2};v_1,t_2}(X), ((i_1,i_2), \id) \ra \dd X^{i_{3}}_\bv \\
        -  \int_{[\bs,\bt]} \la \bS_{s_{1},s_{2};v_{1},v_{2}}(X), ((i_1,i_2), \id) \ra \dd X^{i_{3}}_\bv
    - \int_{[\bs,\bt]} \la \bS_{s_1,v_2;v_1,t_2}(X), ((i_1,i_2), \id) \ra \dd X^{i_{3}}_\bv.  
    \end{multline*}
    Summarizing this by setting 
    \[
    \cL^{\bs\bt}_\bv =  \la \bS_{s_{1},s_{2};v_1,t_2}(X), ((i_1,i_2), \id) \ra 
        -  \la \bS_{s_{1},s_{2};v_{1},v_{2}}(X), ((i_1,i_2), \id) \ra 
    -  \la \bS_{s_1,v_2;v_1,t_2}(X), ((i_1,i_2), \id) \ra,  
    \]
    finishes the proof. 
\end{proof}

\begin{rem}
Proposition \ref{prop:recursive rel for some permutations} only gives an integral equation for the signature up to level $n=3$. For $n\geq 4$, deriving such an equation would require further insight into the algebraic structure which is beyond the scope of the current paper. 
\end{rem}


\subsection{Invariance properties}\label{subsec:invariance}

Among the central and elementary properties of the path signature, let us mention the following ones:
\begin{enumerate}[wide, labelwidth=!, labelindent=0pt, label= \textbf{(\roman*)}]
\setlength\itemsep{.05in}
    \item  \emph{Translation invariance:}  For $a\in \RR^d$ and a path $x:[0,1]\rightarrow \RR^d $ one has $ S(x+a)=S(x)$.
    \item \emph{Re-parametrization invariance:} For a non-decreasing function $\phi:[0,1]\rightarrow [0,1]$ with $\phi(0)=0$ and $\phi(1)=1$ one has $S(X\circ \phi)_{0,1}=S(X)_{0,1}$.
\end{enumerate}

Heuristically, translation invariance tells us that the signature ``does not see'' the absolute level at which the path is operating, but only sees relative differences. Reparametrization invariance tells us that the signature does not care about different speeds at which we might run through our path. These two features have been crucial in assessing the usefulness of the signature method in machine learning and other areas, see e.g. \cite{ChevyKrom2016,LyonsMcL2022}. 

We will now discuss these properties in the setting of the 2D signature. They rely on the choice of differential structure to work with for the concept of a 2D signature. Consequently, we will explore and elaborate on them in relation to the differential $\dd X$ as introduced in the Notation \ref{not:2d-diff}.


\subsubsection{Translation invariance}
\label{ss:translate_invariance}

When working with the 2D signature over a $\RR^d$-valued field $X:[0,T]^2 \rightarrow \RR^d $, there are three different ways to look at translation; one is to shift the field with a constant vector $a\in \RR^d$, another is to shift the field by 1D paths, only depending on either the first variable or the second. More precisely, let $x^i:[0,T]\rightarrow \RR^d$ for $i=1,2$ be two paths, and set $\tau X_{t_1,t_2}:=X_{t_1,t_2}+x^1_{t_1}+x^2_{t_2} +a$. Then, given that we work with the differential structure $\dd X$ defined in Notation 
\ref{not:2d-diff} for the 2D signature $S(X)$, we observe that for every $i\in \{1,\ldots,d\}$ we have 
\[
\dd^i (\tau X) = \dd^i  X,  
\]
and so it follows that 
\begin{equation*}
    \bS(\tau X)=\bS(X). 
\end{equation*}

\begin{rem}
This is a rather striking property of the 2D signature and sheds light on the type of properties in a field $X$ captured by the signature. It is complementary to the classical path signature, and for feature extraction purposes, depending on the task at hand, it might be beneficial to also include the path signature of the paths $t_1\mapsto X(t_1,u)$ and $t_2\mapsto X(u,t_2)$ for some choices of $u$, as the 2D signature will not capture information provided from this. In contrast, the 2D signature will provide us information about how much "small changes in one direction are affected by small changes in the other direction", and then the iterated integrals provide us with a systematic (partially ordered) comparison of this information over a field $X$.   
\end{rem}


\subsubsection{Invariance to stretching}

While there is no direct analogue of the re-parameterization invariance of the path signature in the 2D setting, the closest thing would be to consider invariance to stretching. 

\begin{defn}
We say that a continuously differentiable function $\phi:[0,T]^2\rightarrow [0,T]^2$  is a {\em stretching} if 
\[
\phi(t_1,t_2)=(\phi_1(t_1),\phi_2(t_2)),
\]
where  $\phi_i:[0,T]\rightarrow [0,T]$ are monotone functions with the property that $\phi_i(0)=0$ and $\phi_i(T)=T$ for $i=1,2$.
\end{defn}

For a field $X:[0,T]^2\rightarrow \RR^d$ and a stretching $\phi$ we 
define now the stretched field $X^\phi (t_1,t_2)=X(\phi(t_1,t_2)).$ By the chain rule, for all $i\in \{1,\ldots ,d\}$ we have that 
\[
\dd^i (X^\phi) = (\dd^i X)(\phi) \,\phi'_1\, \phi'_2. 
\]
Integrating over a rectangle $[\bs,\bt]$ and applying change of variable, it is readily checked that
\begin{equation}\label{eq:integral over reparameter}
    \int_{\bs}^\bt \dd^i X^\phi = \int_{\phi(\bs)}^{\phi(\bt)} \dd^i X. 
\end{equation}
In particular, recalling that $\bT$ stands for the tuple $(T,T)$ we have  $\int_{\mathbf{0}}^\bT \dd^i X^\phi = \int_{\mathbf{0}}^{\bT} \dd^i X$. Iterating~\eqref{eq:integral over reparameter} over the word $w=(i_1,..,i_n)$ we further deduce that 
\[
\bS_{\bs\bt}(X^\phi)= \bS_{\phi(\bs),\phi(\bt)}(X), 
\]
and conclude that the 2D signature over $[0,T]^2$ is invariant to stretching.


\subsubsection{Equivariance to rotation}

We will in this section explore how rotations of the field $X$ affect the 2D signature.  To this end, we begin with a definition of rotating a field. 

\begin{defn}\label{def:rotation}
A counterclockwise coordinate rotation by angle $\theta$ around the origin is defined by
\[
\phi_\theta (t_1,t_2)=(t_1\cos\theta-t_2 \sin\theta, t_1\sin\theta +t_2\cos \theta).
\]
Given a continuous field $X:[-1,1]^2\rightarrow \RR^d  $, its clockwise rotation%
\footnote{Note that the counterclockwise rotation
of the parameters induces a \emph{clockwise} rotation
of the field.}
by $\theta$ is defined by 
\begin{equation}\label{eq:rotated X}
    X^\theta (t_1,t_2) := X(\phi_\theta(t_1,t_2)). 
\end{equation}
\end{defn}

\begin{rem}
We will here consider fields defined on the standard square centered in zero, $[-1,1]^2$, as the rotation provided above is around $(0,0)$. This guarantees that when $\theta=\pi/2$ the field is still well defined over $[-1,1]^2$. 
\end{rem}

In the next proposition, we will show a certain equivariance of
the $\id$ Signature with respect to rotations by multiples of $\pi/2$ of
an image over $[-1,1]^2$.

\begin{prop}
Let $X$ be a field in $\cC^2([-1,1]^2;\RR^d)$ and let $(w,\nu)$ be a word in  $\hat{\cW}_n$ for $\nu\in \Sigma_{\{1,\ldots,n\}}$. Let $\rho_n \in \Sigma_{\{1,\ldots,n\}}$ denote the reversal permutation, i.e. where $\rho_n(i)=n-i+1$ for $i=1,\ldots,n$. Then the following three identities holds: 
\begin{enumerate}[wide, labelwidth=!, labelindent=0pt, label= \textbf{(\roman*)}]
\setlength\itemsep{.05in}
\item
The rotation by $\theta=\pi/2$, satisfies (recall the notation for the right-action on words from Notation \ref{not:shuffle-on-words} )
\[
 \la \bS_{-\mathbf{1}\mathbf{1}}(X^{\pi/2}),(w,\nu)\ra
 =(-1)^n
 \la \bS_{\mathbf{-1}\mathbf{1}}(X),
     (w_{\nu^{-1}\circ \rho_n},
     \nu^{-1} \circ \rho_n) \ra. 
\]

\item The rotation by $\theta=\pi$, satisfies
\[
 \la \bS_{-\mathbf{1}\mathbf{1}}(X^{\pi}),(w,\nu)\ra
 =
 \la
     \bS_{\mathbf{-1}\mathbf{1}}(X),
     (w_{ \rho_n^{-1} \circ \nu \circ \rho_n }, \rho_n^{-1} \circ \nu \circ \rho_n)
 \ra. 
\]

\item The rotation by $\theta=3\pi/2$, satisfies
\[
 \la \bS_{-\mathbf{1}\mathbf{1}}(X^{3\pi/2}),(w,\nu)\ra
 =(-1)^n
 \la \bS_{\mathbf{-1}\mathbf{1}}(X),
  (w_{\rho_n^{-1} \circ \nu^{-1}}, \rho_n^{-1} \circ \nu^{-1})
 \ra. 
\]
\end{enumerate}
Here, $X^\theta$ is defined in \eqref{eq:rotated X}. 
\end{prop}

\begin{proof}

It is readily checked from Definition  \ref{def:rotation}  that a $90$ degrees rotation is given by the coordinate change $(t_1,t_2)\mapsto X^{\pi/2}(t_1,t_2)= X(-t_2,t_1)$.  
By elementary calculus rules, in particular repeated change of variables, and using the fact that for $a<b$ we have $\int_b^a g = -\int_a^b g$ for some integrable function $g$, one can check that the following relation holds for a sequence of integrable functions $\{g^i\}_{i=1}^n$
    \begin{equation}
        \int_{\Delta^n_{s,t}} \prod_{i=1}^n g^i(-r^i)\dd r^i = \int_{\Delta^n_{-t,-s}} \prod_{i=1}^n g^i(u^{n-i+1})\dd u^i
    \end{equation}
    Furthermore, we use that $\dd (X^{\pi/2})^{w_k}_{(r_1,r_2)}= -\dd X^{w_k}_{-r_2,r_1}$, and find the following 
    \begin{equation*}
    \begin{aligned}
         \la \bS_{\bs\bt}(X^{\pi/2}),(w,\nu)\ra &= \int_{\Delta^n_{s_1,t_1}} \int_{\Delta^n_{s_2,t_2}\nu^{-1}} \prod_{k=1}^n \dd (X^{\pi/2})^{w_k}_{(r_1^k,r_2^k)} 
         \\
         &=(-1)^n\int_{\Delta^n_{s_1,t_1}} \int_{\Delta^n_{s_2,t_2}} \prod_{k=1}^n \dd X^{w_k}_{(-r_2^{\nu(k)},r_1^k)} 
             \\
         &=(-1)^n
         \int_{\Delta^n_{s_1,t_1}} 
         \int_{\Delta^n_{s_2,t_2}} 
         \prod_{k=1}^n \dd X^{ w_{\nu^{-1}(k)} }_{(-r_2^k,r_1^{\nu^{-1}(k)})} 
         \\
           &=(-1)^n
           \int_{\Delta^n_{s_1,t_1}} 
           \int_{\Delta^n_{-t_2,-s_2} } 
           \prod_{k=1}^n \dd X^{w_{\nu^{-1}(k)}}_{(u_2^{n-k+1},r_1^{\nu^{-1}(k)})} 
            \\
           &=(-1)^n\int_{\Delta^n_{s_1,t_1}} \int_{\Delta^n_{-t_2,-s_2} } 
           \prod_{\ell=1}^n \dd X^{w_{\nu^{-1}(n-\ell+1)}}_{(u_2^{\ell},r_1^{\nu^{-1}(n-\ell+1)})}  \\
           &=
           (-1)^n\la \bS_{(-t_2,s_1),(-s_2,t_1)}(X),
           (w_{\nu^{-1}\circ\rho_n}, \nu^{-1} \circ \rho_n)
           \ra.  
\end{aligned}
\end{equation*}
Note that, the number $-s_2$ is now inserted in the spot where $t_1$ usually is, and similarly for the rest of the variables, reflecting the rotation of the domain. 
And in particular, choosing $[\bs,\bt]=[-1,1]^2$ we see that 
\[
 \la \bS_{-\mathbf{1}\mathbf{1}}(X^{\pi/2}),(w,\nu)\ra =(-1)^n \la \bS_{\mathbf{-1}\mathbf{1}}(X),
 (w_{\nu^{-1}\circ\rho_n}, \nu^{-1} \circ \rho_n)
 \ra. 
\]
This proves item (i).

Now, the items (ii) and (iii) follow from
iterated application of (i)
and noting that the following two relations hold
\begin{align*}
    (\nu^{-1} \circ \rho_n)^{-1} \circ \rho_n
    =
    \rho_n^{-1} \circ \nu \circ \rho_n, \qquad  (\rho_n^{-1} \circ \nu \circ \rho_n)^{-1} \circ \rho_n
    =
    \rho_n^{-1} \circ \nu^{-1}.
\end{align*}
where we have used that $\rho_n = \rho_n^{-1}$ for the second relation. 
\end{proof}

\subsubsection{On homotopy invariance}
\label{sssect:homotopy}

As we
see in \Cref{ss:shuffle},
there are some parallels
to the iterated integrals defined in 
\cite{horozov2015noncommutative}.
The latter are shown to be homotopy invariant,
which is \emph{not} the case for the integrals defined here.
\begin{example}\label{ex:homotopy}
    Consider the map
    \begin{align*}
        H_t(r_1,r_2) = (t \sin(\pi r_2) r_1, t \sin(\pi r_1) r_2).
    \end{align*}
    It is a homotopy between the zero map and the map
    \begin{align*}
        X_{r_1,r_2} := (\sin(\pi r_2) r_1, \sin(\pi r_1) r_2),
    \end{align*}
    which preserves the boundary on $[0,1]^2$.
    Now, the integrals of the zero map are all $0$.
    If the signature were homotopy invariant, we would in particular have
    \begin{equation}\label{eq:zero cond}
        \la \bS_{\mathbf{0},\mathbf{1}}(X),((1,2),\id)\ra =0. 
    \end{equation}
    In order to compute the left hand side of \eqref{eq:zero cond}, observe from our \Cref{not:2d-diff} that 
    \begin{align*}
        \dd^1 X_{r_1,r_2} &= \pi \cos(\pi r_2) \dd r_1\dd r_2
        \\
          \dd^2 X_{r_1,r_2} &= \pi \cos(\pi r_1) \dd r_1\dd r_2.
    \end{align*}
    Therefore, relation \eqref{d5} and some elementary computations show that 
    \begin{align*}
       \la \bS_{\mathbf{0}\mathbf{1}}(X), ((1,2), \id) \ra 
       &=
       \pi \int_0^1 ds_1 \int_0^1 ds_2\ \sin(\pi s_2) s_1 \cos(\pi s_1) \\
       &=
       \pi \left( \int_0^1 ds_1\ s_1 \cos(\pi s_1)  \right)
       \left( \int_0^1 ds_2\ \sin(\pi s_2) \right) \\
       &=
       \pi \left( - \frac2{\pi^2} \right)
       \left( \frac2\pi \right) \\
       &=
       - \frac{4}{\pi^2} \not= 0.
    \end{align*}
    Therefore relation \eqref{eq:zero cond} is not fulfilled, which concludes the lack of homotopy invariance. 
\end{example}
\Cref{ex:homotopy} points to a fundamental difference between Horozov's approach in \cite{horozov2015noncommutative} and ours. Indeed, the integrals of Horozov differ from ours in particular
in the form of the differentials being integrated: ours are second-order, his are first order.
Even when using first first order differentials,
the integrals are, in general, not homotopy invariant.
That they are in \cite{horozov2015noncommutative}
stems from the fact
that the ambient dimension there (our $d$) is equal to $2$,
and in this case all homotopies are \emph{thin} homotopies.


\subsection{Continuity of the 2D signature}

Continuity concerning the signal is an essential part of rough paths theory as well as classical analysis (see e.g.~\cite[Prop.~2.8]{FV}). We derive this type of continuity property here for the 2D signature of sufficiently smooth fields. 
 
\begin{prop}
\label{prop:decay}
Let $X,\tilde{X}$ be two fields in $\cC^2([0,T]^2;\RR^d)$. Consider the corresponding 2D signatures, $ \bS(X)$ and $\bS(\tilde{X})$, given in Definition \ref{def:extended Sig}.  Let also $(w,\tau)\in \hat{\cW}_n$, where $\hat{\cW}_n$ is given as in Definition \ref{def:2d words}.  Then the following estimate holds for any  $(\bs,\bt)\in \Delta^2_{[0,T]^2}$: 
\begin{equation}
\label{eq:local lip}
       | \la \bS_{\bs\bt}(X)-\bS_{\bs\bt}(\tilde{X}),(w,\tau)\ra| \leq \frac{T^{2n}}{(n!)^2} \sum_{k=1}^n
       \lp \prod_{i=1}^{k-1}\|X^{w_i}\|_{\cC^2} \prod_{j=k+1}^n  \|\tilde{X}^{w_j}\|_{\cC^2} \rp \|X^{w_k}-\tilde{X}^{w_k}\|_{\cC^2}.
\end{equation}
In particular, the mapping $X\mapsto \bS_{\bs\bt}(X)$ is locally Lipschitz. 
\end{prop}

\begin{proof}
Observe that for two sequences of numbers, $\{a_i\}_{i=1}^n$ and $\{b_i\}_{i=1}^n$, we have that 
\begin{equation}\label{f1}
    \prod_{i=1}^n a_i -\prod_{i=1}^n b_i = \prod_{i=1}^{n-1}a_i (a_n-b_n)+ \left(\prod_{i=1}^{n-1} a_i-\prod_{i=1}^{n-1}b_i\right) b_n. 
\end{equation}
Now consider a generic word $(w,\tau)\in \hat{\cW}_n$  and $(\br^1,\ldots ,\br^n)\in \Delta^n_{[\bs,\bt]}$. Recalling Notation \ref{not:2d-deriv} for partial derivatives and invoking relation~\eqref{f1}, we have 
\begin{multline*}
        \left|\prod_{i=1}^n \partial_{12} X^{w_i}_{r_1^i,r_2^{\tau(i)}} -\prod_{i=1}^n \partial_{12} \tilde{X}^{w_i}_{r_1^i,r_2^{\tau(i)}}\right|  \\
    \leq \sum_{k=1}^n \left| \prod_{i=1}^{k-1}\partial_{12} X^{w_i}_{r_1^i,r_2^{\tau(i)}} \prod_{j=k+1}^n \partial_{12} \tilde{X}^{w_j}_{r_1^j,r_2^{\tau(j)}} \right|\left| \partial_{12} X^{w_k}_{r_1^k,r_2^{\tau(k)}}-\partial_{12} \tilde{X}^{w_k}_{r_1^k,r_2^{\tau(k)}} \right|. 
\end{multline*}
This leads us to conclude that 
\[
      \left|\prod_{i=1}^n \partial_{12} X^{w_i}_{r_1^i,r_2^{\tau(i)}} -\prod_{i=1}^n \partial_{12} \tilde{X}^{w_i}_{r_1^i,r_2^{\tau(i)}}\right| \leq \sum_{k=1}^n\prod_{i=1}^{k-1}\|X^{w_i}\|_{\cC^2} \prod_{j=k+1}^n  \|\tilde{X}^{w_i}\|_{\cC^2}  \|X^{w_k}-\tilde{X}^{w_k}\|_{\cC^2}, 
    \]
Inserting this into the signature definition, and using that 
\[
\int_{\Delta^n_{[\mathbf{0},\bT]}}1 \dd \br^1\cdots \dd \br^n = \frac{T^2n}{(n!)^2}, 
\]
we obtain 
\begin{align*}
     | \la \bS_{\bs\bt}(X)-\bS_{\bs\bt}(\tilde{X}),(w,\tau)\ra|\leq& \int_{\Delta^n_{[\mathbf{0},\bT]}}  \left|\prod_{i=1}^n \partial_{12} X^{w_i}_{r_1^i,r_2^{\tau(i)}} -\prod_{i=1}^n \partial_{12} \tilde{X}^{w_i}_{r_1^i,r_2^{\tau(i)}}\right|  \dd \br^1\cdots \dd \br^n 
     \\
     \leq& \frac{T^{2n}}{(n!)^2} \sum_{k=1}^n
       \prod_{i=1}^{k-1}\|X^{w_i}\|_{\cC^2} \prod_{j=k+1}^n  \|\tilde{X}^{w_j}\|_{\cC^2}  \|X^{w_k}-\tilde{X}^{w_k}\|_{\cC^2},
\end{align*}
which is our claim \eqref{eq:local lip}. This finishes the proof. 
\end{proof}


\subsection{Products of iterated integrals -- 2D shuffle}
\label{ss:shuffle}

Our interpretation of the shuffle property is that a product of two elements in the signature should be expressed as a linear combination of elements in the signature. As already mentioned in Remark \ref{rem:two issues of ID sig}, one of the main reasons for introducing the notion of the full 2D signature is related to the shuffle property. We will now state this relation in full generality. Notice that a similar result appears in \cite[Proposition 1.3]{horozov2006noncommutative} without proof. We include a complete argument here for the sake of clarity. 
We note that the two-dimensional
shuffle algebra has also recently been investigated
in \cite{schmitz2024free},
in particular in its relation to the
sums signature of \cite{diehl2022two}.

Before stating our main result on shuffle products, let us rephrase the shuffle notation in \eqref{eq:shuffle-permut}  and $(\sigma\star \tau)\circ \rho^{-1}$ in \eqref{eq:prod of delta} directly in terms of the permutations $\sigma,\tau$. This will shorten some of our notation in the computations throughout the section. 

\begin{nota}\label{not:ee}
Recall \Cref{def:shuffle} about permutations. Then consider 
\begin{equation}
    \sigma\in \Sigma_{\{1,\ldots,n\}}, \quad \mathrm{and} \quad \tau \in \Sigma_{\{n+1,\ldots,n+k\}}
\end{equation}
    We call $\sh(\si,\tau)$ the following set of permutations of $\{1,\ldots,n+k\}$
\begin{align*}
    \sh(\si,\tau) :=  \{ \rho\in \Sigma_{\{1,\ldots,n+k\}} \mid \rho\,\, \textrm{does not change the order of $\tau$ and $\sigma$} \}.
\end{align*}
    Note that this can equivalently be defined as 
\begin{align*}
    \sh(\si,\tau) :=  \{ (\sigma \ast \tau) \circ \rho^{-1} \mid \rho \in \sh(n,k) \} ,
\end{align*}
where we have used that any $\rho \in \sh(\si,\tau)$ can be written as $\rho = (\si\ast \tau)\circ \hat{\rho}^{-1}$ for a given element $\hat{\rho}\in \sh(n,k)$. 
\end{nota}
\begin{rem}
    Thanks to \Cref{not:ee} one can rephrase Proposition \ref{prop:shuffle operators} in a slightly more compact way. That is, for $\si\in \Sigma_{\{1,\ldots,n\}}$ and $\tau\in \Sigma_{\{1,\ldots,n'\}}$ we have 
   \begin{equation}\label{eq:fffff}
       \int_{\Delta^n_{s,t}\si} \star \int_{\Delta^{n'}_{s,t}} = \sum_{\phi\in \sh(\si,\tau)} \int_{\Delta^{n+n'}_{s,t}\phi}. 
    \end{equation}
\end{rem}
We are now ready to state the shuffle property for the full signature $\bS$.

\begin{prop}[Shuffle relation]\label{prop:shuffle relation}
Let $X$ be a field in $\cC^2([0,T]^2;\RR^d)$ and $\bS(X)$ be its 2D signature (Definition \ref{def:extended Sig}). We consider two words, $(w,\nu)\in \hat{\cW}_n$ and $(w',\nu')\in \hat{\cW}_{n'}$ (Definition \ref{def:2d words}).
Then for $(\bs,\bt)\in \Delta^2_{[0,T]^2}$ we have 
\begin{equation}
\label{eq:shuffle relation final}
    \langle \bS_{\bs\bt}(X),(w,\nu)\rangle 
   \langle \bS_{\bs\bt}(X),(w',\nu')\rangle  
   =\sum_{\tau\in \sh(\id_n,\id_{n'})} \sum_{\lambda\in \sh(\nu,\nu')}      \langle \bS_{\bs\bt}(X),([ww']_\tau,\lambda\circ\tau) \rangle ,
\end{equation}
where we recall Notation \ref{not:shuffle-on-words} for the concatenation $[ww']$. In~\eqref{eq:shuffle relation final}, we also recall that for a word $w$ with $|w|=n$ and a permutation $\tau \in \Sigma_{\{1,\ldots,n\}}$ we set $w_\tau=(w_{\tau(1)},\ldots ,w_{\tau(n)})$.
\end{prop}

\begin{proof}
For two generic words $(w,\nu)\in \hat{\cW}_n$ and $(w',\nu')\in \hat{\cW}_{n'}$, let us denote by $\cG_{\bs\bt}$ the quantity of interest, i.e.
\[
    \cG_{\bs\bt} :=  \langle \bS_{\bs\bt}(X),(w,\nu)\rangle \langle \bS_{\bs\bt}(X),(w',\nu')\rangle.  
\]
Owing to relation \eqref{eq:rep of 2d sig with simplex permutation} one can write 
\[
\cG_{\bs\bt} = \left( \int_{\Delta^n_{s_1,t_1}}\starI \int_{ \Delta^n_{s_2,t_2}\nu^{-1}} \right) \star \left( \int_{\Delta^{n'}_{s_1,t_1}}\starI \int_{ \Delta^{n'}_{s_2,t_2}\nu'^{-1}} \right) \prod_{i=1}^n \dd X^{w_i}_{\br} \prod_{j=1}^{n'} \dd X^{w'_j}_{\br^j},
\]
which, by \Cref{lem:interchange},
is equal to
   \[
\cG_{\bs\bt}
= \left( \int_{\Delta^n_{s_1,t_1}}\star \int_{\Delta^{n'}_{s_1,t_1}} \right) \starI \left( \int_{ \Delta^n_{s_2,t_2}\nu^{-1}} \star \int_{ \Delta^{n'}_{s_2,t_2}\nu'^{-1}} \right) \prod_{i=1}^n \dd X^{w_i}_{\br} \prod_{j=1}^{n'} \dd X^{w'_j}_{\br^j}
\] 
We now apply relation \eqref{eq:fffff} separately for the integrals over $[s_1,t_1]$ and $[s_2,t_2]$.  This yields
\begin{equation}\label{G rel}
    \cG_{\bs\bt} = \sum_{\rho\in \sh(\id_n,\id_{n'}) } \sum_{\sigma\in \sh(\nu^{-1},\nu'^{-1})} \cG_{\bs\bt} ^{\rho,\sigma}, 
\end{equation}
where we define 
\begin{equation}\label{G def}
    \cG_{\bs\bt} ^{\rho,\sigma} = \int_{\Delta^{n+n'}_{s_1,t_1}\rho \times \Delta^{n+n'}_{s_2,t_2}\sigma} \prod_{i=1}^{n+n'} \dd X^{\hw_i}_{\br^i}, 
\end{equation}
and we have set $\id_n = \id_{\{1,\ldots ,n\}}$ and  $\hw = [ww']$. To express every element in \eqref{G rel}
as an element of the 2D signature like \eqref{eq:rep of 2d sig with simplex permutation}, we now replace the simplexes $\Delta^{n+n'}_{s_1,t_1}\rho$ by $\Delta^{n+n'}_{s_1,t_1}$ in the definition~\eqref{G def} of $\cG_{\bs\bt} ^{\rho,\sigma}$. To this aim, we simply declare that $r_1^k = r_1^{\rho_k}$ and perform the corresponding change in the multi-index $\hw$. Namely, we define a new word $w^\rho$ by setting 
\[
    \hw_{\rho^{-1}} = (\hw_{\rho^{-1}_1},\ldots, \hw_{\rho^{-1}_{n+n'}}).
\]
This simple change of variables enables us to write 
\begin{equation}
\label{f2}
    \cG^{\rho,\sigma}_{\bs\bt} = 
\int_{\Delta^{n+n'}_{s_1,t_1}\times \Delta^{n+n'}_{s_2,t_2}\sigma} 
\prod_{i=1}^{n+n'} \dd X^{\hw_i}_{r_1^{i},r_2^{\rho_{i}^{-1}}}  = 
\int_{\Delta^{n+n'}_{s_1,t_1}\times \Delta^{n+n'}_{s_2,t_2}\rho\circ \sigma} \prod_{i=1}^{n+n'} \dd X^{\hw_{\rho^{-1}(i)}}_{r_1^{i},r_2^{i}}.  
\end{equation}
And so we can rewrite~\eqref{f2} resorting to the representation in \eqref{eq:rep of 2d sig with simplex permutation}, in order to obtain 
\begin{equation}\label{eq:aaa}
\cG^{\rho,\sigma}_{\bs\bt} = \la \bS_{\bs\bt}(X),(\hw_{\rho^{-1}},\sigma^{-1}\circ\rho^{-1})\ra.    
\end{equation}
Gathering \eqref{eq:aaa} and \eqref{G rel}, we have thus obtained 
\begin{equation*}
      \cG_{\bs\bt} = \sum_{\rho\in \sh(\id_n,\id_{n'}) } \sum_{\sigma\in \sh(\nu^{-1},\nu'^{-1})}\la \bS_{\bs\bt}(X),(\hw_{\rho^{-1}},\sigma^{-1}\circ\rho^{-1})\ra.
\end{equation*}
Next, notice that $\sigma\in \sh(\nu^{-1},\nu'^{-1})$ iff $\sigma^{-1}\in \sh(\nu,\nu')$. Setting $\lambda=\sigma^{-1}$  and similarly $\tau=\rho^{-1}$ above we end up with 
\[
    \cG_{\bs\bt} = \sum_{\tau\in \sh(\id_n,\id_{n'}) } \sum_{\lambda\in \sh(\nu,\nu')}\la \bS_{\bs\bt}(X),(\hw_\tau,\lambda\circ \tau)\ra, 
\]
which is exactly the relation in \eqref{eq:shuffle relation final}, and so concludes our proof. 
\end{proof}

As an application of the shuffle relation in Proposition \ref{prop:shuffle relation}, we will now derive a formula for products of rectangular increments. The proposition below should be seen as an integration by parts formula for 2D-increments. 

\begin{prop}
Under the same conditions and using the same notation as in Proposition~\ref{prop:shuffle relation}, recall that the rectangular increments $\square X$ are given in \Cref{def:squareincrement}. Then for $w=(i_1,\ldots,i_n)\in \cW_n$ and $(\bs,\bt)\in \Delta^2_{[0,T]^2}$ we have 
\begin{equation}
\label{eq:prod to sig}
    \prod_{k=1}^n \square_{\bs\bt} X^{w_k} 
        = 
    \sum_{\nu,\nu'\in \Sigma_{\{1,\ldots,n\}}} \la \bS_{\bs\bt}(X),(w_\nu,\nu')\ra .
    \end{equation}
\end{prop}
\begin{proof}
    The relation is obvious for $n=1$. We now proceed by induction. Namely, assume relation \eqref{eq:prod to sig}  holds true for any $n\geq 1$. We consider a word $w=(\hw,i_{n+1})$  in $\cW_{n+1}$, where $\hw=(i_1,\ldots i_n)\in \cW_n$ and $i_{n+1}\in \{1,\ldots,d\}$. Denote by $U_{\bs\bt}^{n+1}$ the quantity of interest for the induction procedure, that is, $U^{n+1}_{\bs\bt}=\prod_{k=1}^{n+1} \square_{\bs\bt}X^{i_k}$. Then we trivially have that 
    \[
    U^{n+1}_{\bs\bt}=\left(\prod_{k=1}^{n} \square_{\bs\bt}X^{i_k}\right)\square_{\bs\bt}X^{i_{n+1}}.
    \]
Owing to our induction assumption and the fact that $\square_{\bs\bt} X^{i_k}$ is an element of the signature, this yields 
\begin{equation}
\label{eq:aaa3}
    U^{n+1}_{\bs\bt} = \sum_{\nu,\nu'\in \Sigma_{\{1,\ldots,n\}}} \la \bS_{\bs\bt}(X),(\hw_\nu,\nu')\ra \la \bS_{\bs\bt}(X),(i_{n+1},\id_1)\ra .
\end{equation}
Observe that instead of using \eqref{eq:shuffle relation final} to express every term in \eqref{eq:aaa3}, it is slightly more convenient to invoke \eqref{G rel}. We get 
\begin{equation*}
    \la \bS_{\bs\bt}(X),(\hw_\nu,\nu')\ra \la \bS_{\bs\bt}(X),(i_{n+1},\id_1)\ra  = \sum_{\rho\in \sh(\nu^{-1},\id_{1}) } \sum_{\sigma\in \sh(\nu'^{-1},\id_1)} \int_{\Delta^{n+1}_{s_1,t_1}\rho \times \Delta^{n+1}_{s_2,t_2}\sigma} \prod_{k=1}^{n+1} \dd X_{\br^k}^{i_k}. 
\end{equation*}
Plugging this identity into \eqref{eq:aaa3} we obtain 
\begin{equation}
\label{eq:dd}
         U^{n+1}_{\bs\bt} 
         = \sum_{\nu,\nu'\in \Sigma_{\{1,\ldots,n\}}} \sum_{\rho\in \sh(\nu^{-1},\id_{1}) } \sum_{\sigma\in \sh(\nu'^{-1},\id_1)} \int_{\Delta^{n+1}_{s_1,t_1}\rho \times \Delta^{n+1}_{s_2,t_2}\sigma} \prod_{k=1}^{n+1} \dd X_{\br^k}^{i_k}. 
\end{equation}
In addition, it is easily seen that 
    \begin{equation*}
        \bigcup_{\nu'\in \Sigma_{\{1,\ldots,n\}}}\sh(\nu'^{-1},\id_1) = \Sigma_{\{1,\ldots,n+1\}}. 
    \end{equation*}
    Therefore, one can recast \eqref{eq:dd} as 
    \begin{equation}\label{eq:e}
        U^{n+1}_{\bs\bt} 
         = \sum_{\nu\in \Sigma_{\{1,\ldots,n\}}} \sum_{\rho\in \sh(\nu^{-1},\id_{1}) } \sum_{\tau\in \Sigma_{\{1,\ldots,n+1\}}} \int_{\Delta^{n+1}_{s_1,t_1}\rho \times \Delta^{n+1}_{s_2,t_2}\tau} \prod_{k=1}^{n+1} \dd X_{\br^k}^{i_k}. 
    \end{equation}
We now proceed as in the proof of Proposition \ref{prop:shuffle relation}. Namely, we want to replace the integral over $\Delta^{n+1}_{s_1,t_1}\rho$ by an integral over $\Delta^{n+1}_{s_1,t_1}$. To this end, we perform the same change of variable, $r_1^k=r^{\rho_k}_1$, which leads to a corresponding permutation $\mu$ of the word $[\hw,i_{n+1}]$. We let the reader then check that \eqref{eq:e} can then be read as 
\begin{equation*}
        U_{\bs\bt}^{n+1} 
        = \sum_{\mu,\nu\in \Sigma_{\{1,\ldots,n+1\}}} \int_{\Delta^{n+1}_{s_1,t_1} \times \Delta^{n+1}_{s_2,t_2}\nu} \prod_{k=1}^{n+1} \dd X_{\br^k}^{w_{(\mu^{-1}\circ \nu)(k)}}
        =\sum_{\nu,\nu'\in \Sigma_{\{1,\ldots,n+1\}} }\int_{\Delta^{n+1}_{s_1,t_1} \times \Delta^{n+1}_{s_2,t_2}\nu'} \prod_{k=1}^{n+1} \dd X_{\br^k}^{w_{\nu(k)}}.
\end{equation*}
This proves the induction step for \eqref{eq:prod to sig}, which allows us to conclude. 
\end{proof}


\section{The symmetrized signature}
\label{sec:symetrizedSig}

Up to now, we have introduced two types of 2D signatures for a $\cC^2$-field $X$, namely $\bS^{\id}(X)$ and $\bS(X)$. The 2D signature $\bS^{\id}(X)$ satisfies a Chen-type relation but fails to satisfy the shuffle relation. For $\bS(X)$ the situation is the opposite. In the current section, we introduce another notion of 2D signature which enjoys both Chen and shuffle relations, at the price of integrating over the whole rectangle $[s_2,t_2]^k$ in the second variable (in contrast to the simplex $\Delta^{k}_{s_2,t_2}$). Let us start by defining this new type of 2D signature.

Recall that for $n\geq 0$ we defined the set $\cW_n$ of words of length $n$ in \eqref{def:words} and the set $\hat{\cW}_n$ of extended words in Definition \ref{def:2d words}. Furthermore, the 2D signature $\bS(X)$ is introduced in Definition \ref{def:extended Sig}. 

\begin{defn}
Let $X$ be a field in $\cC^2([0,T]^2;\RR^d)$. For $w\in \cW_n$ and $(\bs,\bt)\in \Delta^2_{[0,T]^2}$ we set 
\begin{equation}\label{eq:def symmetrized sig}
        \la \bS^{\sym}_{\bs\bt}(X),w \ra := \sum_{\nu\in \Sigma_{\{1,\ldots,n\}}} \la \bS_{\bs\bt}(X),(w,\nu)  \ra .
\end{equation}
\end{defn}

\begin{rem}
Notice that for $(\bs,\bt)\in \Delta^2_{[0,T]^2}$ one can also see $\bS^{\sym}_{\bs\bt}(X)$ as an element of $\cT((\RR^d))$. This kind of representation follows closely the lines of Section \ref{ssec:Sigalg}. 
\end{rem}

\begin{rem}
Combining \eqref{eq:def symmetrized sig} with the representation \eqref{eq:rep of 2d sig with simplex permutation} of $\bS(X)$ and Lemma \ref{lem:simplex to hypercubes}, we find for $(\bs,\bt)\in \Delta^2_{[0,T]^2}$
\begin{equation}\label{eq:rep of sym sig as monom}
          \la \bS^{\sym}_{\bs\bt}(X),w \ra = \sum_{\nu\in \Sigma_{\{1,\ldots,n\}}} \int_{\Delta^n_{s_1,t_1}}\starI \int_{\Delta^n_{s_2,t_2}\nu^{-1}} \prod_{i=1}^n \dd X^{w_i}_{\br^i} = \int_{\Delta^n_{s_1,t_1}}\starI \int_{[s_2,t_2]^n} \prod_{i=1}^n \dd X^{w_i}_{\br^i}.
\end{equation}
This formula will be useful for future computations. 
\end{rem}


\subsection{Chen's relation}
\label{sec:chenrelation}

As mentioned above, the symmetrized 2D signature enjoys important algebraic properties. We now prove partial versions of Chen's relation in this context.

\begin{prop}\label{prop:sym chen}
Let $X\in \cC^2([0,T]^2;\RR^d)$ and consider $\bS^{\sym}(X)$ defined by \eqref{eq:def symmetrized sig}. In addition, pick a word $w\in \cW_n$ together with $(\bs,\bt)\in \Delta^2_{[0,t]^2}$ and $s_1\leq u_1 \leq t_1$. Then we have 
\begin{equation}\label{eq:chen for symm}
            \la \delta^1_{u_1} \bS^{\sym}_{\bs\bt}(X),w\ra = \sum_{k=1}^{n-1} \la \bS^{\sym}_{\bs(u_1,t_2)}(X),(w_1\cdots w_k)\ra  \la  \bS^{\sym}_{(u_1,s_2)\bt}(X),(w_{k+1}\cdots w_n)\ra. 
\end{equation}
\end{prop}

\begin{proof}
Start from relation \eqref{eq:rep of sym sig as monom} and apply the $\delta^1$ operation from Notation \ref{not:delta} on both sides of the identity. Thanks to \eqref{eq:ff}, we get 
\begin{equation*}
        \la \delta^1_{u_1} \bS^{\sym}_{\bs\bt}(X),w\ra = \sum_{k=1}^{n-1} \lp \int_{\Delta^{k}_{s_1,u_1}}\star \int_{\Delta^{n-k}_{u_1,t_1}} \rp \starI \int_{[s_2,t_2]^n} \prod_{i=1}^n \dd X^{w_i}_{\br^i }. 
\end{equation*}
Now we decompose $[s_2,t_2]^n = [s_2,t_2]^{k} \times [s_2,t_2]^{n-k}$ and apply Fubini's theorem. We get 
\begin{align*}
    \la \delta^1_{u_1} \bS^{\sym}_{\bs\bt}(X),w\ra 
    =& \sum_{k=1}^{n-1} \lp \int_{\Delta^{k}_{s_1,u_1}}\starI \int_{[s_2,t_2]^k}\rp  \star \lp  \int_{\Delta^{n-k}_{u_1,t_1}} \starI \int_{[s_2,t_2]^{n-k}} \rp  \prod_{i=1}^n \dd X^{w_i}_{\br^i }\\
    =&  \sum_{k=1}^{n-1} \la \bS^{\sym}_{\bs(u_1,t_2)}(X),(w_1\cdots w_k)\ra  \la  \bS^{\sym}_{(u_1,s_2)\bt}(X),(w_{k+1}...w_n)\ra \, ,
\end{align*}
where the last identity is a simple consequence of \eqref{eq:def symmetrized sig}. The proof of \eqref{eq:chen for symm} is now complete. 
\end{proof}


\subsection{Shuffle relations for the symmetrized 2D signature}

We have seen in Proposition \ref{prop:shuffle relation} that the 2D signature $\bS(X)$ satisfies a shuffle-type relation. The same is true for the symmetrized 2D signature.
\begin{prop}
\label{prop:2dsymmetrizedshuffle}
Under the same conditions as in Proposition \ref{prop:sym chen}, let $w \in \cW_n$ and $w' \in \cW_{n'}$ for $n,n'\geq 1$.
Then for all $(\bs,\bt)\in \Delta^2_{[0,T]^2}$ we have 
\begin{equation}
\label{eq:sym shuffle rel}
        \la\bS^{\sym}_{\bs\bt}(X),w\ra \la\bS^{\sym}_{\bs\bt}(X),w'\ra 
        = \sum_{\sigma \in \sh(n,n')} \la \bS^{\sym}_{\bs\bt}(X),[ww']_\sigma \ra 
        = \la\bS^{\sym}_{\bs\bt}(X),w \dsqcup w'\ra \, ,
\end{equation}
where $w \dsqcup\, w'$ is inductively defined in \eqref{eq:shuffle}. 
\end{prop}

\begin{proof}
For generic words $w \in \cW_n$ and $w'\in \cW_{n'}$, we use $\eqref{eq:rep of sym sig as monom}$ in order to write 
\begin{eqnarray*}
           \la\bS^{\sym}_{\bs\bt}(X),w\ra \la\bS^{\sym}_{\bs\bt}(X),w'\ra  = \lp \int_{\Delta^n_{s_1,t_1}}\starI \int_{[s_2,t_2]^n} \prod_{i=1}^n \dd X^{w_i}_{\br^i} \rp \lp \int_{\Delta^{n'}_{s_1,t_1}}\starI \int_{[s_2,t_2]^{n'}} \prod_{i=1}^{n'} \dd X^{w'_i}_{\br^i}\rp \, .
\end{eqnarray*}
    By \Cref{lem:interchange}
    \begin{equation}
          \la\bS^{\sym}_{\bs\bt}(X),w\ra \la\bS^{\sym}_{\bs\bt}(X),w'\ra  = \lp
          \left( \int_{\Delta^n_{s_1,t_1}} \star \int_{\Delta^{n'}_{s_1,t_1}} \right) \starI \int_{[s_2,t_2]^{n+n'}} \prod_{i=1}^{n+n'} \dd X^{[ww']_i}_{\br^i}\rp 
    \end{equation}
    Now apply \Cref{prop:shuffle operators}
    to get the claimed result.
\end{proof}


\section{The signature determines the field}
\label{sec:Funcapprox}


In this section we show that the 2D signature
of a  (time-enhanced) field determines the field.
%
Our proof is based on the Stone-Weierstrass theorem, and relies upon lifting the path to an extended path where the two `time' variables are included.
A similar trick is typically used in the classical universal approximation theorem for path signatures.
We note that after establishing the ``point-separating''
property of the signature, together with the shuffle relation derived in Proposition \ref{prop:shuffle relation},  one can deduce
that the linear functionals on the signature
are dense in an appropriate space of continuous functionals
on fields (see for example \cite[Section 10]{giusti2022topological}). For notational sake, the analysis in this section will be restricted to fields defined on $[0,1]^2$. the trivial extension to $[0,T]^2$ is left to the reader. 

We begin this section with a technical lemma that will be central in subsequent proofs. 
\begin{lem}
    \label{lem:Psi}
    Given $\delta \in (0,1), \eps > 0$ there exists
    a univariate polynomial $\Psi$ such that
    \begin{align*}
        \Psi(y) &\in [0,1], &&\text{for } y \in [0,1], \\
        |\Psi(y)| &< \eps, &&\text{for } y \in [0,\delta/2], \\
        |\Psi(y)-1| &< \eps, &&\text{for } y \in [\delta,1].
    \end{align*} 
\end{lem}
\begin{proof}
    The proof idea is from \cite{brosowski1981elementary}.
    We recall Bernoulli's inequality, 
    \begin{align*}
        (1 + x)^n \ge 1 + nx, \quad \text{for } x \ge -1, n \in \NN.
    \end{align*}
Pick $k \in \mathbb N_{\ge 1}$ such that the following relation holds
\begin{align*}
  1 < k\delta < 2.
\end{align*}
Given $m \in \mathbb N_{\ge 1}$ define the function 
\begin{align*}
  \Phi_m(y) := (1-y^m)^{k^m}.
\end{align*}
Then we see that 
\begin{align*}
  \Phi_m(y) \in [0,1] , \quad \textrm{for all}\quad  y \in [0,1].
\end{align*}
Furthermore, for any  $y \in [0,\frac\delta2]$,  using Bernoulli's inequality we see that 
\begin{align*}
  \Phi_m(y) \ge 1 - k^m y^m > 1 - \left(\frac{k\delta}2\right)^m.
\end{align*}
We observe that since $k\delta/2 < 1$, this tends to $1$
(independent of $y \in [0,\frac\delta2]$) as $m\rightarrow \infty$.
Now for  the case $y \in [\delta,1]$, it is readily checked that 
\begin{align*}
  \Phi_m(y)
  &= \frac{1}{k^m y^m} (1-y^m)^{k^m} k^m y^m \\
  &\le \frac{1}{k^m y^m} (1-y^m)^{k^m} (1 + k^m y^m) \\
  &\leq  \frac{1}{k^m y^m} (1-y^{2m})^{k^m}
  \le \frac{1}{k^m y^m}
  \le \frac{1}{(k\delta)^m}.
\end{align*}
Here we used Bernoulli's inequality for the second inequality.
We observe that since $k\delta > 1$, this tends to $0$
(uniformly in $y \in [\delta,1]$) as $m\rightarrow \infty$.
Combining the two cases, we see that we can choose $m$ large enough such
that $\Phi := \Phi_m$ satisfies, 
\begin{align}
  \label{h}  \Phi(y) &\in [0,1] \text{ for } y \in [0,1] \\
   \label{i} \Phi(y) &\in [1-\epsilon,1] \text{ for } y \in [0,\frac\delta2] \\
   \label{j} \Phi(y) &\in [0,\epsilon] \text{ for } y \in [\delta,1].
\end{align}
Then $\Psi := 1- \Phi$ satisfies the required bounds.
\end{proof}

The next step is to consider a family of functions that is dense in the space of continuous functions on $[0,1]^2$ with zero boundary. To this end, 
let $\mathcal{F}$ denote the $\RR$-algebra spanned by the functions,
\begin{align}\label{eq: def of cF}
    (s_1,s_2) \mapsto s_1^{2m + n} s_2^{2n + m}, \quad m, n \ge 0.
\end{align}
With this definition at hand, we prove another technical lemma that will allow us to prove that $ \cF$ is a dense subset of $C([0,T]^2;\RR)$ with zero boundary. 

\begin{lem}
    \label{lem:g}
    For every $b > 0$ small enough, there exists a  $ b' \in (0, b)$
    such that for all $a > 0$ there is a function
    $g_{a,b',b} \in \mathcal{F}$ such that
    \begin{align}
       \label{eq:gnorm} \|g\| &\le 1, && \\
        \label{eq:gnorm 2}|g(s_1,s_2)| &< a, \qquad &&\text{for } s_1 \in [0,b'] \text{ or } s_2 \in [0,b'], \\
        \label{eq:gnorm 3}|g(s_1,s_2)-1| &< a, \qquad &&\text{for } s_1,s_2 \in [b,1], 
    \end{align}
    where $\|g\|$ stands for the $\sup$ norm in $[0,1]^2$. 
\end{lem}
\begin{proof}
    Consider the function 
    \begin{align*}
        P(s_1,s_2) := s_1 s_2 \in \mathcal{F}.
    \end{align*}
    Then $P(s_1,s_2) \in [0,1]$ for $s_1,s_2 \in [0,1]$.
    For the given $b$ , let $\delta$ be the number defined by  
    \begin{align*}
        \delta := \inf_{s_1,s_2 \in [b,1]} P(s_1,s_2), 
    \end{align*}
    and observe that $\delta = b^2$. Now pick  $b' < b$ such that the following relation holds
    \begin{align*}
        \sup_{s_1,s_2 \in ([0,1]\times [0,b']) \cup ([0,b']\times [0,1])} P(s_1,s_2) < \delta/2.
    \end{align*}
Next we consider a polynomial $\Phi$ satisfying the requirements of \Cref{lem:Psi}, for    $\eps = a$. We let  $g := \Phi\circ P $. It is readily checked from the definition of the space proposed in \eqref{eq: def of cF} that  $g\in \cF$. Therefore, relations \eqref{h}-\eqref{j} translate into 
\begin{align*}
    &|g(s_1,s_2)|\in [0,1], \quad \forall \, (s_1,s_2)\in [0,1]^2
    \\
        &|g(s_1,s_2)|\leq a, \quad \forall \, (s_1,s_2)\in ([0,1]\times [0,b']) \cup ([0,b']\times [0,1])
        \\
                &|1-g(s_1,s_2)|\leq a, \quad \forall \, (s_1,s_2)\in [0,1]^2
\end{align*}
   Otherwise stated, our claims \eqref{eq:gnorm}-\eqref{eq:gnorm 3} are proved. 
\end{proof}

\begin{lem}
    \label{lem:dense}
    Let $\cF$ be the vector space defined by \eqref{eq: def of cF}. Then
    $\mathcal{F}$ is dense in
    the subspace
    \begin{align*}
        V := \{ f \in C([0,1]^2, \RR) \mid f(0,s_2) = f(s_1,0) = 0, s_1,s_2 \in [0,1] \},
    \end{align*}
    of continuous functions vanishing at the ``left-bottom'' boundaries.
\end{lem}
\begin{proof}
The strategy for the proof is a subtle variation on classical Stone-Weierstrass type arguments. Namely on every sub-rectangle $[b',1]$, Stone-Weierstrass theorem will give the desired conclusion. Then close to the axes we shall resort to a polynomial interpolation given by \Cref{lem:g}. 

    First let us prove that $\mathcal{F}$ separates points in $(0,1]^2$.
    Indeed, 
    if
    \begin{equation*}
    s_1^{2} s_2     = q_1^{2} q_2 ,    \qquad
    s_1     s_2^{2} = q_1     q_2^{2},
    \end{equation*}
    and $s_1, s_2, q_1, q_2 \in (0,1)$ then
    \begin{align*}
    \frac{s_1}{q_1} = \frac{q_2^2}{s_2^2} = \frac{s_1^4}{q_1^4},
    \end{align*}
    hence
    \begin{align*}
        s_1 = q_1, \quad s_2 = q_2.
    \end{align*}
    Furthermore, the constant function is contained in $\mathcal{F}$
    and it is closed under multiplication.
    Hence, by the Stone-Weierstrass theorem,
    the functions are dense in $C([\epsilon,1]^2, \RR)$,
    for every $\epsilon > 0$.

    We now start our interpolation argument close to the axes. That is consider $f \in V$ and $\eps > 0$ given.
    Choose $\delta > 0$ such that
    \begin{align}\label{eq:ineq f}
        |f(s_1,s_2)| < \eps, \quad \text{for } s_1\in [0,\delta] \text{ or } s_2 \in [0,\delta].
    \end{align}
    Let $b' < b$ be given as in 
    \Cref{lem:g},
    for $b := \delta$.
    Denote $\mathcal{F}_{b'}$ the restriction of $\mathcal{F}$ to $[b',1]^2$.
    By the classical Stone-Weierstrass theorem,
    there exists $\hat f_{\eps,b'} \in \mathcal{F}_{b'}$ such that
    \begin{align}\label{p}
        ||f - \hat f_{\eps,b'}||_{ C([b',1]^2, \RR) } < \eps. 
    \end{align}
    Fix an extension $f_{\eps,b'} \in \mathcal{F}$ of $\hat f_{\eps,b'}$ to $[0,1]^2$.
    Define
    \begin{align*}
        M := \sup_{s_1,s_2 \in [0,b']} |f_{\eps,b'}(s_1,s_2)| + 1.
    \end{align*}
    Consider $g_{\eps/M,b',b} \in \mathcal{F}$ as in \Cref{lem:g}.
    Then, $h := f_{\eps,\delta} \cdot g_{\eps/M,b',b} \in \mathcal{F}$ satisfies the following three properties: 
\begin{enumerate}[wide, labelwidth=!, labelindent=0pt, label= \textbf{(\roman*)}]
\setlength\itemsep{.05in}
        \item 
        For  $s_1 \in [0,b'] $ or $s_2 \in [0,b']$:
        \begin{align*}
            |h(s_1,s_2)|
            &=
            |f_{\eps,\delta}(s_1,s_2)| \cdot |g_{\eps/M,b',b}(s_1,s_2)| \\
            &\le
            M \cdot \eps / M = \eps.
        \end{align*}
        Since $f$ verifies \eqref{eq:ineq f} for $s_1\in [0,b']$ or $s_2\in [0,b']$, we get that $|h(s_1,s_2)-f(s_1,s_2)|\leq 2\epsilon$. 
        \item
        For $s_1 \in [b',b], s_2 \in [b',1]$ or $s_2 \in [b',b], s_1 \in [b',1]$
        \begin{align*}
            |h(s_1,s_2)|
            &=
            |f_{\eps,\delta}(s_1,s_2)| \cdot |g_{\eps/M,b',b}(s_1,s_2)| \\
            &\le
            2 \eps \cdot ||g_{\eps/M,b',b}||_\infty 
            \le 2 \eps.
        \end{align*}
        
        \item
        For $s_1,s_2 \in [\delta,1]$
        \begin{align*}
            &|h(s_1,s_2) - f(s_1,s_2)| \\
            &\quad\le
            |(f_{\eps,\delta}(s_1,s_2) - f(s_1,s_2)) \cdot g_{\eps/M,b',b}|
            +
            |f(s_1,s_2) (g_{\eps/M,b',b} - 1)| && \\
            &\quad\le
            \eps + ||f|| \eps/M.
        \end{align*}
    \end{enumerate}
    Now $\eps$ was arbitrary (and $M \ge 1$), and so it follows that  $\mathcal{F}$ is dense in $V$.
\end{proof}

We now turn to the main result of the section, which states that an extended version of the signature characterizes a smooth field up to an equivalent class.

\begin{thm}
Consider the following equivalence relation
on $C^2([0,T]^2;\RR^d)$:
\begin{align}\label{eq:sim rel}
    X \sim Y \iff \partial_{12} X = \partial_{12} Y.
\end{align}
    For $X \in C^2([0,1]^2;\RR^d)$ consider the ``lift'' to
    $\hat X \in  C^2([0,1]^2;\RR^{2+d})$
    \begin{align*}
        \hat X_{s_1,s_2} := (s_1^2 s_2, s_1 s_2^2, X_{s_1,s_2}).
    \end{align*}
    Then, for $X, Y \in C^2([0,1]^2;\RR^d)$ we have $X \sim Y$ if and only if
    \begin{align*}
        \bS_{\boo,\mathbf{1}}(\hat X) = \bS_{\boo,\mathbf{1}}(\hat Y).
    \end{align*}
\end{thm}
\begin{proof}
    It suffices to show the case $d=1$, i.e.
    \begin{align*}
        \hat X_{s_1,s_2} = (\hat X_{s_1,s_2}^{(1)}, \hat X_{s_1,s_2}^{(2)}, \hat X_{s_1,s_2}^{(3)}) = (s_1^2 s_2, s_1 s_2^2, X_{s_1,s_2}) \in \RR^3.
    \end{align*}
   It is clear that 
    \begin{align*}
        \la \bS_{\boo \bs}(\hat X), 1 \ra 
        =
        s_1^2 s_2, \quad \mathrm{and } \quad 
        \la \bS_{\boo \bs}(\hat X), 2 \ra
        =
        s_1 s_2^2.
    \end{align*}
    Then, by the shuffle relation it follows that 
    \begin{align*}
        \la \bS_{\boo \bs}(\hat X), 1^{\shuffle m} \shuffle 2^{\shuffle n} \ra
        =
        s_1^{2m + n} s_2^{2n + m}.
    \end{align*}
    Hence, by \Cref{lem:dense}, for every $f\in\mathcal F$, $\eps > 0$ there exists
    $\phi_\eps \in T(\RR^d)$ such that with
    \begin{align*}
        \Phi_\eps(s_1,s_2) := 
        \la \bS_{\boo \bs}(\hat X), \phi_\eps \ra,
    \end{align*}
    we have $||f - \Phi_\eps|| < \eps$.
    Note that 
    \begin{align*}
        \la \bS_{\boo \mathbf{1}}(\hat X), (1^{\shuffle m} \shuffle 2^{\shuffle n}) 3 \ra
        =
        \int_0^1 ds_1 \int_0^1 ds_2 \ s_1^{2m + n} s_2^{2n + m} \partial_{12} X_{s_1,s_2}.
    \end{align*}
    If $\bS_{\boo \mathbf{1}}(\hat X) = \bS_{\boo \mathbf{1}}(\hat Y)$ then,
    for every $f \in \mathcal F$ we have 
    \begin{align*}
        \int_0^1 ds_1 \int_0^1 ds_2 \ f(s_1,s_2) \partial_{12} X_{s_1,s_2}
        =
        \int_0^1 ds_1 \int_0^1 ds_2 \ f(s_1,s_2) \partial_{12} Y_{s_1,s_2}.
    \end{align*}
    By the usual approximation of the Dirac delta we hence get
    \begin{align*}
        \partial_{12} X(s_1,s_2) = \partial_{12} Y(s_1,s_2), \quad \text{for } s_1,s_2 \in (0,1].
    \end{align*}
    By continuity of $\partial_{12} X, \partial_{12} Y$,
    we hence get
    \begin{align*}
        \partial_{12} X = \partial_{12} Y.
    \end{align*}
We have thus proved that whenever $\bS_{\mathbf{0}\mathbf{1}}(\tilde{X})=\bS_{\mathbf{0}\mathbf{1}}(\tilde{Y})$ we have $\tilde{X}\sim \tilde{Y}$ according to \eqref{eq:sim rel}. The other implication being trivial, this finishes the proof. 
\end{proof}

\section{Concluding remarks and open problems}
\label{sec:conclussion}

We have proposed several 2D extensions of the path signature arising in the theory of rough paths \cite{LCL}. In contrast to the multiparameter extensions proposed by \cite{giusti2022topological,lee2023random}, the 2D-signatures we consider here are constructed over fields $X:[0,T]^2 \rightarrow \RR^d $, using the mixed partial derivative $\frac{\partial^2 }{\partial t_1 \partial t_2}X$  of the field as the driving signal. As such we have invariance to path perturbations
(\Cref{ss:translate_invariance}),
and thus the 2D-signatures may serve as a complement to the classical path signature when working with image data. We show that our (full) signature satisfies shuffle relations (\Cref{prop:shuffle relation}), and the $\id$-signature satisfies a type of Chen relation based on convolutional products (\Cref{Prop:Horizontal and vertical chen}).
%
As the current paper only deals with sufficiently smooth, i.e., twice continuously differentiable fields $X$, we have not considered rough path-type questions, such as an extension theorem for the signature, or investigations of probabilistic aspects of the 2D signature constructed from random fields. This is a direction of work which we are currently investigating. 
 
\bibliographystyle{plain}

\end{document}